\documentclass[oneside]{amsart}
\usepackage{dsfont}
\usepackage[utf8]{inputenc}
\usepackage{amsmath,amsthm,amssymb}
\usepackage{mathrsfs} 
\usepackage{graphicx, xcolor}
\usepackage{subcaption}
\usepackage{caption} 
\usepackage{subfiles}
\usepackage{newverbs} 
\usepackage{array} 
\usepackage{tikz}
\usepackage{braket}
\usepackage{enumerate}
\usepackage{mathtools}

\colorlet{green}{black}

\usepackage[colorlinks,citecolor=blue,linkcolor=red!70!black]{hyperref}
\usepackage[nameinlink]{cleveref}

\usepackage[marginpar=4cm,left=5cm,right=5cm]{geometry}

\usepackage{amssymb} 
\usepackage{graphicx} 

\newverbcommand{\CMRverb}{\tiny\color{blue}}{}
\newverbcommand{\GRverb}{\tiny\color{teal}}{}
\newverbcommand{\STverb}{\tiny\color{cyan}}{}
\newverbcommand{\Pverb}{\tiny\color{violet}}{}
\newverbcommand{\Averb}{\tiny\color{brown}}{}

\theoremstyle{plain}
\newtheorem{THM}{Theorem}[section]

\newtheorem{PROP}[THM]{Proposition}
\newtheorem{LEM}[THM]{Lemma}
\newtheorem{COR}[THM]{Corollary}

\newtheorem*{CLAIME}{Claim}

\theoremstyle{definition}
\newtheorem{DEF}[THM]{Definition}
\newtheorem{RMK}[THM]{Remark}
\newtheorem{EX}[THM]{Example}

\newtheorem*{NOTN}{Notation}

\DeclareMathOperator{\card}{card}
\DeclareMathOperator{\Leb}{Leb}

\newcommand{\dfn}{\mathrel{\mathop:}=}

\newcommand{\C}{\mathbb{C}}

\renewcommand{\P}{\mathcal{P}}

\newcommand{\R}{\mathbb{R}}

\newcommand{\Z}{\mathbb{Z}}
\newcommand{\B}{\mathcal{B}}
\newcommand{\T}{\mathbb{T}}
\newcommand{\e}{\epsilon}
\newcommand{\HH}{\text{H}}
\newcommand{\N}{\mathbb{N}}
\newcommand{\of}{\circ}
\newcommand{\ve}{\epsilon}

\newcommand{\A}{\mathcal{A}}
\newcommand{\Flow}{\mathcal{F}}
\newcommand{\Isom}{\operatorname{Isom}}
\newcommand{\Lie}{\operatorname{Lie}}
\newcommand{\ad}{\operatorname{ad}}
\newcommand{\htop}{h_{\operatorname{top}}}

\newcommand{\topo}{\operatorname{top}}

\renewcommand{\bar}{\overline}
\renewcommand{\|}{|\!|} 

\renewcommand{\emptyset}{\varnothing} 
 
\newcommand{\semitop}{\operatorname{semi}} 

\title{Slow Entropy and Variational Dynamical Systems}
\author[Cheng]{Minhua Cheng}
\address{DEPARTMENT OF MATHEMATICS, UNIVERSITY OF UTAH,
 SALT LAKE CITY, UT 84112}
 \email{cheng@math.utah.edu}

\author[Ospina]{Carlos Ospina}
\address{DEPARTMENT OF MATHEMATICS, UNIVERSITY OF UTAH,
 SALT LAKE CITY, UT 84112}
 \email{ospina@math.utah.edu}

\author{Kurt Vinhage}
\address{DEPARTMENT OF MATHEMATICS, UNIVERSITY OF UTAH,
 SALT LAKE CITY, UT 84112}
\email{vinhage@math.utah.edu}

\author[Zhai]{Yibo Zhai}
\address{DEPARTMENT OF MATHEMATICS, UNIVERSITY OF UTAH,
 SALT LAKE CITY, UT 84112}
 \email{zhai@math.utah.edu}

\begin{document}

\begin{abstract}
    We define variational properties for dynamical systems with subexponential complexity, and study these properties in certain specific examples. By computing the value of slow entropy directly, we show that some subshifts are not variational, while a class of interval exchange transformations are variational.
\end{abstract}

\maketitle

\section{Introduction}

The metric and topological entropies for measure-preserving and topological dynamical systems are often the first and most important invariants to study. These notions of entropy are numbers assigned to a dynamical system which assigns a complexity based on the exponential growth rate of the number of distinguishable orbit segments.

We will investigate foundational properties of the {\it slow entropy-type invariants} introduced by Katok and Thouvenot \cite{KatokThouvenot}, with an emphasis on establishing some results accepted as folklore, as well as some features of the usual entropy that do {\it not} pass to the invariants at subexponential rates. 

We recall the usual definitions of entropy in dynamical contexts in \Cref{sec:classical}. Entropy as a dynamical invariant stems from its formulation in information theory by Shannon. In the smooth setting, entropy is connected with the study of Lyapunov exponents due to the Pesin and Ledrappier-Young entropy formulas. These formulas and perspectives have played crucial roles in seemingly unrelated areas such as thermodynamical formalism, progress on the Furstenberg $(\times 2, \times 3)$-conjecture and its generalizations, and the superrigidity phenomena for higher-rank Lie groups (the Zimmer program).

We refer the reader to \cite[Sections 3.1, 4.3]{hasselblatt-katok} for a more detailed introduction to the classical entropy theory, and \cite{50yearsofentropy} for a review of the history of the standard entropy theory and more context on the history hinted at here.

\subsection{The variational principle}The metric and topological entropies for continuous transformations are linked via the {\bf  variational principle}:

\begin{equation}
    \label{eq:vprinciple-classic}
    \htop(T) = \sup_{\mu \in \mathcal M^T} h_\mu(T),
\end{equation}
where $\mathcal M^T$ is the set of $T$-invariant Borel probability measures. If this supremum is achieved, a measure for which $h_\mu(T) = \htop(T)$ is called a {\bf measure of maximal entropy} (or MME).

The variational principle is the fundamental connection between the two entropy theories, and features of a measure of maximal entropy can reveal many properties of the underlying dynamical system. For instance, for geodesic flows in negative curvature, it is conjectured that the measure of maximal entropy is the Liouville measure if and only if the underlying manifold is locally symmetric. This is known for surfaces \cite{Katok1982}, but it remains open as the Katok entropy conjecture in higher dimensions.

For general flows and diffeomorphisms on surfaces, a measure of maximal entropy has fractional dimension, and the closer it is to the dimension of the manifold, the more ``equidistributed'' the divergence is in the space.


\subsection{Slow entropy invariants}
The {\bf slow entropy} of a dynamical system is a class of invariants which can describe subexponential growth rates such as polynomial or logarithmic, and indeed can be given specific numerical values which are invariants of uniformly continuous or measure-preserving conjugacy, depending on the category. The terminology slow entropy was introduced by Katok and Thouvenot \cite{KatokThouvenot}, but others have studied it under various other names, including {\bf measure-theoretic complexity} \cite{Ferenczi0}. In the setting of shift spaces using {\bf language complexity} (or simply the {\bf complexity}, see \Cref{sec:complexity}). These correspond to the metric and topological slow entropy, respectively.

We discuss the definitions and basic properties in  \Cref{sec:defs}, but introduce some notation here, where the notation $p_\chi(n) = n^\chi$ represents the family of {\bf polynomial scales}:

\begin{itemize}
    \item if $h_{\mu,p_\chi} = d$, then the number of orbit types of length $n$ which can be distinguished by $\mu$ is approximately $n^d$, and
    \item if $h_{\topo,p_\chi} = d$, then the number of orbit types of length $n$ which are topologically distinguished is approximately $n^d$.
\end{itemize}

By analogy, the classical entropy can be seen as the slow entropy with respect to the family of {\bf exponential scales} $e_\chi(n) = e^{\chi n}$, and one may consider in general entropies at a more general family of scales $a_\chi$.

Slow entropy or measure-theoretic complexity has many applications in classification questions. One important characterization is those systems with minimal complexity. This was first proved by Ferenczi in \cite{Ferenczi0}, where the result was phrased using measure-theoretic complexity. We present a proof of this theorem in  \Cref{sec:ferenczi}.

\subsection{Variational systems}

It is natural to ask whether \Cref{eq:vprinciple-classic} holds for slow entropy. We will see that the answer in general is no (\Cref{thm:sturmian}), motivating the following definition. We refer the reader to \Cref{sec:defs} for any undefined terms.

\begin{DEF}
    Let $(X,d)$ be a metric space and $T : X \to X$ be a transformation. We say that $T$ is {\bf variational at the family of scales $\set{a_\chi}$} if

    \[h_{\topo,a_\chi} = \sup_{\mu \in \mathcal{M}^T} h_{\mu,a_\chi}.\]

    We say that $T$ is {\bf strongly variational at the family of scales $\set{a_\chi}$} if there exists a unique measure $\mu_0$ for which 

    \[ h_{\topo,a_\chi} = h_{\mu_0,a_\chi}.\]
\end{DEF}

Variational properties of several known examples can be deduced from from existing work:

\begin{itemize}
    \item Transitive translations on compact abelian groups are strongly variational at all scales (Ferenczi's Theorem \cite{Ferenczi0}, \Cref{thm:Ferenczi}).
    \item All continuous transformations of compact spaces are variational at exponential scale (This is the classical variational principle, see eg, \cite[Theorem 4.5.3]{hasselblatt-katok}).
    \item Uniformly hyperbolic dynamical systems are strongly variational at exponential scale (Existence and uniqueness of MMEs for uniformly hyperbolic follows from classical works of Bowen and Margulis).
    \item Transitive unipotent flows are strongly variational at polynomial scale (\cite{ParabolicFlows}).
    \item Some smooth systems obtained from combinatorial constructions are not variational (\cite{BanKunWei1},\cite{BanKunWei2}).
\end{itemize}

In this paper, we establish the following theorems towards understanding which systems are variational:

\begin{THM}
    \label{thm:sturmian}
    Sturmian subshifts and Denjoy circle transformations are not variational at polynomial scale.
\end{THM}

\begin{THM}
\label{thm:3iet-main}
    There exists a full Hausdorff dimension subset of 3-IETs which are strongly variational at polynomial scale.
\end{THM}

 \Cref{thm:sturmian} is proved in \Cref{sec:sturmian} and  \Cref{thm:3iet-main} is proved in  \Cref{sec:CompSlowEntr}. We also provide a precise description of a full Hausdorff dimension set of 3-IETs which are variational at polynomial scale.

\subsection{Future directions}

The proof of \Cref{thm:3iet-main} requires strong Diophantine conditions to compute the slow entropy of the 3-IETs. While it is likely that these conditions can be relaxed, the proof suggests that intermediate behavior for slow entropy is possible. In particular, we believe that some 3-IETs are {\it not} variational.

This is perhaps less shocking after noting that interval exchanges have discontinuities. However, by adding roof functions with controlled singularities, such transformations are realized as first return maps for surface flows with stationary points. It is therefore natural to ask whether the surface flows are variational. The slow entropy of some surface flows was computed in \cite{SpecialFlows}, so a computation of their topological slow entropy would determine their variational properties.

In \Cref{sec:ergodic-decomp}, we observe that slow entropy does not behave like exponential entropy with respect to ergodic decompositions, and there is little hope to obtain a universal formula. In fact, one may have a system with positive slow entropy for which every ergodic component is Kronecker. It is natural to ask how large the gap between the entropy of the ergodic components and the entropy of the integrated system can become. By the usual formula for entropy given an ergodic decomposition, we know that the gap cannot be exponential.

For strongly variational systems, one can try to discern what information about the underlying system can be learned from properties of the entropy maximizing measure. Particularly, there may be analogs of the Katok entropy conjecture for systems with positive and finite entropy at polynomial scale. Correspondingly, for nonvariational systems, one should be able to identify some erratic divergence of orbits which is seen a the topological level but not detected by measures.

Finally, we note that for non-variational systems, the gap between \\ $\sup_{\mu \in \mathcal M(X)} h_{\mu,a_\chi}$ and $h_{\topo,a_\chi}$ can be very large (\Cref{thm:big-polynomial} and \Cref{thm:big-gap}). The examples we describe here are uniquely ergodic, and the metric slow entropy is 0 at all scales, but the topological slow entropy is very large. It would be interesting to find non-variational systems with many invariant measures, and variational systems which are not strongly variational at a subexponential scale.

\subsection{Organization of the paper} 
In \Cref{sec:defs}, we review three types of invariants for systems with zero entropy at an exponential scale. The first is the topological slow entropy which we denote $h_{\topo,a_\chi}$. This is an invariant of dynamical systems under uniformly continuous conjugations. The second invariant is the metric slow entropy $h_{\mu,a_\chi}$, which is an invariant of measure-preserving systems. The third and last invariant of measure-preserving systems on metric spaces is the semi-topological slow entropy $h_{\semitop,\mu,a_\chi}$. We compare these definitions with the classical topological and metric entropies. In \Cref{sec:coincide}, we show that slow entropy coincides with other entropy and complexity invariants. In particular, \Cref{thm:ExponentialScales} shows that the slow entropies are the classical entropies when $a_{\chi}(n) = e^{\chi n}$, and \Cref{prop:complexity} shows that the topological slow entropy of a subshift captures the usual complexity function studied in those settings.

In \Cref{sec:structuralthms}, we present two results, although not new and generally considered ``folklore,'' they are adaptations of standard arguments in the classical entropy setting to arbitrary scales. Particularly, \Cref{thm:SlowVariationalPrinciple} was proved for first time in \cite{goodwyn} to show that the topological entropy is larger or equal to the metric entropy. We prove that this result is still true without requiring continuity of the system, and also, to include in the inequality the semi-topological entropy. To summarize, we are able to prove, for any scale $a_\chi$, the inequality
\begin{equation*}
    h_{\mu,a_\chi} \leq h_{\semitop,\mu,a_\chi} \leq h_{\topo,a_\chi}.
\end{equation*}
\Cref{sec:ferenczi} is mainly expository. We have parsed \cite[Proposition 3]{Ferenczi0} in \Cref{thm:Ferenczi}, to rewrite it to our notation and for referencing in later results of this paper. We remark that is crucial for proving \Cref{thm:sturmian}. In 
\Cref{sec:sturmian,sec:denjoy}, we prove  \Cref{thm:sturmian}. We compute the semi-topological slow entropy of Sturmian systems, the main result is summarized in \Cref{Sturmian}, and we add a discussion of Denjoy circle transformations to show that these systems are not variational. 
Finally, we discuss the interaction between slow entropy and the ergodic decomposition, see \Cref{sec:ergodic-decomp}. In particular, in \Cref{lemm:HaarGeod} we show that that the metric entropy of the geodesic flow on a flat torus is equal to 1 at polynmial scale, but 0 with respect to any ergodic measure and for any arbitrary scale.

In \Cref{sec:iets}, we introduce the background of interval exchange transformations and prove \Cref{thm:3iet-main}.

\vspace{.3cm}

\noindent {\bf Acknowledgements.} The authors would like to thank Przemyslaw Berk, Jon Chaika, Adam Kanigowski, Scott Schmieding, and Daren Wei for discussions and recommendations on the direction of this paper. The authors also acknowledge the NSF award {\#}1840190, the research training group {\it Algebra, Geometry and Topology}, which provided the space and opportunity for this work to be done.

\section{Definitions}
\label{sec:defs}
Let $f\colon X \to X$ be a measurable transformation on a locally compact metric space.
We denote $(X,d)$ the metric structure, and if $\mu$ is a Radon measure, let $(X,\mu,\B)$ denote the triple determining a measure space and $\sigma$-algebra of Borel sets $\B$. If $\mu$ is $f$-invariant, i.e. $\forall B \in \B$ $\mu(f^{-1}(B))=\mu(B)$, then we say that $(X,\mu,f,\B)$ is a measure preserving system.
We do not assume $\mu$ is ergodic, instead, we mention it whenever it is necessary. However, unless otherwise noted, we will assume that $\mu$ is a probability measure, i.e. $\mu(X) = 1$.

\subsection{Topological Slow Entropy}
\label{sec:top-slow}
For $x \in X$, and $\e > 0$, let $B(x,\e)=\{y \in X: d(x,y)<\e\}$ denote an open ball. For $n$ a non-negative integer and $\e$ a positive real number, it is not hard to check that the map 
\begin{equation*}
    d^n_f(x,y) = \max_{0\le i \le n-1}d (f^i(x), f^i(y))
\end{equation*}
defines a new metric on $X$, called the \textbf{$n$-Bowen metric}, or simply \textbf{Bowen metric}. When $f$ is uniformly continuous, the metric $d_f^n$ is equivalent to $d$. The $(\e,n)${\bf-Bowen ball} is the ball in $d_f^n$ centered at $x$ of radius $\e$, or equivalently the following set:
\begin{equation*}
    B^n_f(x,\e) = \cap_{i=0}^{n-1} f^{-i}\left( B(f^i(x),\e) \right).
\end{equation*}

Let $N_{f,K}(\e,n)$ denote the minimal number of $(\e,n)$-Bowen balls that cover a compact set $K \subset X$. Note that $N_{f,K}(\e,n) < \infty$ when $f$ is uniformly
continuous.

Let $S_{f,K}(\e,n)$ be the maximal number of disjoint $(\e,n)$-Bowen balls that can be arranged with centers in $K$, where all the centers of such a collection of Bowen balls form a \textbf{maximal $(\e,n)$-(Bowen) separating set}.

The following inequalities are important in establishing that a well-defined invariant exists. They are used even in the usual definition of entropy (see, e.g. \cite[Section 3.1.b]{hasselblatt-katok})
\begin{eqnarray}
    S_{f,K}(\e,n) & \le & N_{f,K}(\e,n) \label{eq:top-def1}\\
    N_{f,K}(2\e,n) & \le & S_{f,K}(\e,n). \label{eq:top-def2}
\end{eqnarray}

Note that inequalities \eqref{eq:top-def1} and \eqref{eq:top-def2} still hold when $f$ is not continuous, since $d_f^n$ is still a metric in this case. Uniform continuity is usually used to guarantee the $d_f^n$ is equivalent to $d$.

\begin{DEF}
    A \textbf{scale} $\{a_\chi\}$ is a family of increasing functions, 
    \begin{equation*}
        a_\chi\colon \Z_{>0} \to \R_{>0}
    \end{equation*}
    indexed by $\chi \in \R_{\ge 0}$, such that if $\chi < \chi'$ then $a_{\chi} = o(a_{\chi'})$. 
\end{DEF}

\begin{NOTN}
    In subscripts, $a_\chi$ is used to indicate the scale chosen beforehand, and $\chi$ is used to indicate that the quantity depends on the value of the parameter $\chi$ at the given scale $\{a_\chi\}$.
\end{NOTN}

We think of a scale as a family of functions indexed by $\chi$ that describe the orbit growth. The corresponding slow entropy $h$ means that the orbits grow in time as the function $a_{h}$. If the slow entropy with respect to a given scale is zero (resp. infinity) then, in time $n$ the orbits grow slower (resp. faster) than the sequence $\{a_h(n)\}_{n \in \Z_{>0}}$ for all $h \in \R_+$.

\begin{EX}
    \begin{enumerate}
        \item At exponential scale  $a_\chi(n) = e^{\chi n}$, we will show, in \Cref{thm:ExponentialScales}, that the topological slow entropy (similarly, for metric slow entropy and semitopological slow entropy) and the classical topological entropy are equal.
        \item At polynomial scale  $a_\chi(n) = n^\chi$,  the slow topological entropy is often called polynomial topological entropy. Similar for other scales.
        \item Another example is the logarithmic scale $a_\chi(n) = n(\log n)^\chi$. 
    \end{enumerate}
    These scales have been used, for example, in \cite[Theorem 1.7]{ParabolicFlows}, where it was shown that quasi-unipotent flows have positive polynomial entropy. Previously, \cite[Theorems 1.1 and 1.2]{SpecialFlows} showed that the Korchegin flow has positive polynomial entropy, and the Arnol'd flow has positive logarithmic entropy.
\end{EX}
We define 
\begin{equation*}
    \delta_{f,K,\chi}^N(\e) = \limsup_{n \to \infty} \frac{N_{f,K}(\e,n)}{a_\chi(n)},
\end{equation*} 
and 
\begin{equation*}
    \delta_{f,K,\chi}^S(\e) = \limsup_{n \to \infty} \frac{S_{f,K}(\e,n)}{a_\chi(n)}.
\end{equation*}

\begin{DEF}
The {\bf slow topological entropy} of $f$ for the scale $a_\chi$ is
\begin{align*}
        h_{\topo,a_\chi}(f)=\sup_{K}\lim_{\e \to 0}\left( \sup \left\{ \chi:
    \delta_{f,K,\chi}^N(\e)>0 \right\}\right) \\
    =
    \sup_{K}\lim_{\e \to 0}\left( \sup \left\{ \chi:
    \delta_{f,K,\chi}^S(\e)>0 \right\}\right).
\end{align*}

Here, the supremum is taken over all compact sets $K$.
\label{def:MetricEntropy}
\end{DEF}

Note that this is well-defined by inequalities \eqref{eq:top-def1} and \eqref{eq:top-def2}. Further, one may show that the innermost supremum is decreasing in $\e$, so the limit exists as $\e \to 0$.

\begin{RMK}
\label{rmk:SetPostive}
    We index the family of scales $(a_\chi)$ by the non-negative reals. When taking the supremum over some property $P$ of $\chi$, it can happen that the set $\{\chi:P(\chi) \text{ holds}\}$ is empty. In such a case, we use the convention that the $\sup$ is zero. 
    The set $\left\{ \chi: \delta_{f,K,\chi}^N(\e)>0 \right\}$ is an interval that starts at zero. However, the point $\hat\chi = \sup \left\{ \chi: \delta_{f,K,\chi}^N(\e)>0 \right\}$ may or may not belong to the interval. All this follows because if $\chi>\hat{\chi}$, then by definition we have $\chi \not\in \left\{ \chi: \delta_{f,K,\chi}^N(\e)>0 \right\};$ if $\chi < \hat{\chi},$ again by definition we have $\gamma \in (\chi, \hat{\chi}) $ such that $\delta ^N _{f, K ,\gamma} > 0,$ and thus 
    \begin{equation*}
        \begin{split}
            \limsup_{n\to \infty}  \frac{N_{f,K}(\e,n)}{a_\chi(n)}& = \limsup_{n\to \infty}  \frac{N_{f,K}(\e,n)}{a_{\gamma}(n)} \cdot \frac{a_{\gamma}(n)}{a_\chi(n)} \\
            & \geq \limsup_{n\to \infty}  \frac{N_{f,K}(\e,n)}{a_{\gamma}(n)} >0
        \end{split}
    \end{equation*}
    since $a_{\chi} (n) = o(a_{\gamma}(n))$.

    In conclusion, $\left\{ \chi: \delta_{f,K,\chi}^N(\e)>0 \right\}$ is of the form $[0,\hat{\chi})$ or $[0,\hat{\chi}]$.
\end{RMK}

\begin{RMK}
\label{rmk:discontinuities}
    While many systems on compact metric spaces are continuous, there are certain natural systems which have discontinuities appearing. In this paper we treat the case of 3-IETs, whose discontinuities appear naturally when considering first return maps for Poincar\'{e} sections of flows. Crucially, it is important to note that it still makes sense to consider topological entropy for such systems, but the topological entropy may now depend on the choice of metric on $X$.
\end{RMK}

\subsection{Metric Slow Entropy}

For a probability measure-preserving system $(X,\mu,f,\B)$, consider a finite measurable partition $\P=\{P_1,\dots,P_k\}$. We call each set $P_i$ an {\bf atom} of $\P$. Note that every $x \in X$ defines a coding sequence $(x_s) \dfn (x_s)_{s \in \Z_{\ge 0}}$, where $x_s=j$ if $f^{s}(x) \in P_j$.
For any $x,y \in X$, the {\bf Hamming distance} with respect to the partition $\P$ is the quantity
\begin{equation*}
    \bar{d}_{f,\P}^n(x,y) = 1-\frac{|\{0\le s\le n-1: x_s=y_s\}|}{n},
\end{equation*} 
where $|\cdot|$ is the counting measure. The number $\bar{d}_{f,\P}^n(x,y)$ is the proportion of times for which the orbits of 
$x$ and $y$ lie in different atoms of the partition $\P$ up to time $n$.

For $n \geq 0$ and $\e > 0$, the {\bf $(\e,n)$-Hamming ball} centered at $x \in X$ is the set
\begin{equation*}
    B^n_{f,\P}(x,\e) = \{ y \in X:  \bar{d}_{f,\P}^n(x,y) < \e \}.
\end{equation*}

Next, $F$ represents a finite subset of $X$, and define the number

\begin{equation*}
    S_{f,\P}(\e,n) = \min \left\{ \card(F): \mu\left( \bigcup_{x \in F} B^n_{f,\P}(x,\e) \right) > 1-\e\right\}.
\end{equation*}
For a given scale $a_\chi$, we define
\begin{equation*}
    \delta_{f,\P,\chi}^S (\e)=\limsup_{n\to \infty}\frac{S_{f,\P}(\e,n) }{a_\chi(n)},
\end{equation*}
and the slow metric entropy for the partition $\P$ is 
\begin{equation*}
    h_{\mu,a_\chi,\P}(f)=\lim_{\e \to 0}\left( \sup \left\{ \chi: \delta_{f,\P,\chi}^S (\e) >0 \right\} \right).
\end{equation*}

\begin{DEF}
The \textbf{slow metric entropy} of $f$ with respect to the scale $a_\chi$ is defined as
\begin{equation*}
    h_{\mu,a_\chi}(f) = \sup_{\P}h_{\mu,a_\chi,\P}(f).
\end{equation*}
Here the $\sup$ is taken over all finite measurable partitions of $X$.
\end{DEF}

\subsection{Semi-topological slow entropy}

A notion between topological and metric entropy can be obtained when a natural metric and measure are linked. Let
\begin{equation*}
    S_{\semitop,f}(\e,n) = \min \left\lbrace{\card(F) : \mu\left(\bigcup_{x \in F} B_f^n(x,\e)\right) > 1-\epsilon}\right\rbrace.
\end{equation*}
We similarly let $\delta_{\semitop,f,\chi} ^S (\e) = \displaystyle\limsup_{n \to \infty} \dfrac{S_{\semitop,f}(\e,n)}{a_\chi(n)}$.

\begin{DEF}
    The {\bf semi-topological slow entropy} of $f$ with respect to $\mu$ is

    \[ h_{\semitop,\mu,a_\chi}(f) = \limsup_{\e \to 0}\left( \sup_\chi \left \{ \chi : \delta_{\semitop,f,\chi} ^S (\e) > 0 \right \}\right).\]
\end{DEF}

\section{Slow Entropy as other growth invariants}
\label{sec:coincide}
The slow entropy invariants we have defined are in fact generalizations of the standard classification tools. We describe the connections here.

\subsection{Classical entropy as slow entropy}
\label{sec:classical}

Throughout, we assume that $(X,d)$ is a metric space, $\mu$ is a probability measure on $(X,\B)$, where $\B$ is the Borel $\sigma$-algebra of $(X,d)$, and the $f: X \to X$ is a $\mu$-preserving transformation. 
\begin{DEF}
The {\bf (classical) topological entropy} $h_{\topo}(f)$ is 
\begin{equation}
    h_{\topo}(f) = \sup_{K}\lim_{\e \to 0}\limsup_{n\to\infty}\frac{\log(N_{f,K}(\e,n))}{n}
    \label{eq:TopEntropy}
\end{equation}
\end{DEF}

We refer to the reader to \cite[Section 3.1.b]{hasselblatt-katok} or \cite{Bowen} for an alternative definition using $S_{f,K}(\e,n)$, and a discussion on how the classical topological entropy does not depend on the choice of metric $d$ determining the topology when $X$ is compact and $f$ is continuous.

\begin{DEF}
    The {\bf (classical) metric entropy} of $h_\mu(f)$ is defined as
    \begin{equation}
        h_\mu(f)= \sup_{\P} h_{\mu,\P}(f),
    \end{equation}
    where the supremum is taken over all finite measurable partitions $\P$ and
    \begin{equation}
        h_{\mu,\P}(f) = \lim_{n \to \infty} \frac{1}{n}\HH_\mu(\vee_{i=0}^{n-1} f^{-i}(\P)).
 \label{eq:defofmetricentropy}
    \end{equation}
\end{DEF}
In \Cref{eq:defofmetricentropy}, for any measurable partition $\P$, $$\HH_\mu(\P)\dfn\sum_{P \in \P}-\mu(P)\log(\mu(P)).$$

We have the following well-known theorem in the literature that very few authors proved.

\begin{THM}[Exponential scales in slow entropy]
\label{thm:ExponentialScales}
Let $a_\chi(n) = e^{\chi n}$ and $\mu$ be an ergodic probability measure. Then 
\begin{equation*}
    h_{\mu,a_{\chi}}(f) = h_{\semitop,\mu,a_\chi}= h_\mu(f) \quad \text{and} \quad  h_{\topo,a_\chi}(f) =  h_{\topo}(f).
\end{equation*} 
\end{THM}

\begin{proof} 
 To prove that $h_{\topo,a_\chi} = h_{\topo}$, we claim that for any compact set $K$ and $\hat{\e}>0$,
\begin{equation} 
    \sup\{\chi : \delta^N_{f,K,\chi}(\hat{\e})>0\} = \limsup_{n \to \infty} \frac{\log(N_{f,K}(\hat{\e},n))}{n}. 
\label{eq:ExponentialEquivalence}
\end{equation}

First of all, $\limsup\frac{\log(N_{f,K}(\hat{\e},n))}{n}= \infty$, is equivalent to: there exists $n_i \to \infty $ such that for all $M,\gamma >0$, there exists $n_*,$ such that if $n_i > n_*$, then $\frac{\log(N_{f,K}(\hat{\e},n_i))}{n_i} > M+ \log(1+\gamma)$. This is equivalent to $\frac{N_{f,K}(\hat{\e},n_i)}{e^{Mn_i}}>(1+\gamma)^{n_i}$ for all $n_i>n_*.$ 
Remember that 
\begin{equation}
    \delta^N_{f,K,\chi}(\hat{\e}) = \limsup_{n \to \infty} \frac{N_{f,K}(\hat{\e},n)}{e^{\chi n}}.
    \label{eq:DeltaNEhat}
\end{equation}
So, $\delta_{f,K,M}^N(\hat{\e}) =\infty.$ Since $M>0$ is arbitrary; we conclude that it is equivalent that the left side in \Cref{eq:ExponentialEquivalence} is equal to $\infty.$

Now, assume that $\hat{\chi} = \limsup_{n \to \infty} \frac{\log(N_{f,K}(\hat{\e},n))}{n} < \infty$.

If $\chi > \hat{\chi}$, then the expression in \Cref{eq:DeltaNEhat} is zero, because 
\begin{equation}
 \limsup_{n\to \infty}\frac{N_{f,K}(\hat{\e},n)}{e^{\chi n}}  =  \limsup_{n \to \infty}\frac{N_{f,K}(\hat{\e},n)}{e^{\hat{\chi} n}}\limsup_{n\to \infty}e^{(\hat{\chi} - \chi)n} = 0.
\end{equation}
This implies that
\begin{equation}
    \hat{\chi} \geq  \sup \{\chi: \delta^N_{f,K,\chi}(\hat{\e})>0\}.
    \label{eq:BoundOnSup}
\end{equation}
If $\hat{\chi} = 0$, \Cref{eq:BoundOnSup} is an equality, and it proves \Cref{eq:ExponentialEquivalence}. If $\hat{\chi}>0$,  let $\chi < \hat{\chi}$, and $n_i \to  \infty$ any sequence of positive integers such that $\frac{N_{f,K}(\hat{\e},n_i)}{e^{\hat{\chi} n_i}} \geq c > 0,$ for a positive constant $c$. Then,
we have that 
\begin{equation}
    \frac{N_{f,K}(\hat{\e},n_i)}{e^{\chi n_i}}  e^{(\chi - \hat{\chi})n_i} \geq c.
\end{equation}

Since $e^{(\chi-\hat{\chi})n_i} \to 0$, then $\frac{N_{f,K}(\hat{\e},n_i)}{e^{\chi n_i}} \to \infty.$ And then \Cref{eq:DeltaNEhat} with $\chi<\hat{\chi}$ is equal to $\infty$.
This implies that 
\begin{equation}
    \hat{\chi} = \sup \{\chi: \delta^N_{f,K,\chi}(\hat{\e})>0\}.
\end{equation}

Finally, to prove that $h_{\topo,a_\chi} = h_{\topo}$, assume that $h_{\topo,a_\chi}$ and $ h_{\topo}$ are finite. We focus only on the case when both entropies are finite and leave the case of infinite entropy to the reader. 

It is enough to prove that for arbitrary $\e > 0$, then 
$\left| h_{\topo,a_{\chi}} - h_{\topo} \right| < \e.$

By definition of the entropies, and by triangle inequality, there exists $K \subset X$ compact and $\hat{\e}>0$ sufficiently small, such that
\begin{equation*}
    \begin{split}
        \left| h_{\topo,a_{\chi}} - h_{\topo} \right|& \leq 
         \left| h_{\topo,a_{\chi}} - \sup \{\chi: \delta^N_{f,K,\chi}(\hat{\e})>0\} \right| \\
          +  &  \left| h_{\topo} - \limsup_{n \to \infty} \frac{\log(N_{f,K}(\hat{\e},n))}{n} \right| \\
          +  &  \left| \sup \{\chi: \delta^N_{f,K,\chi}(\hat{\e})>0\} - \limsup_{n \to \infty} \frac{\log(N_{f,K}(\hat{\e},n))}{n}   \right| \\ 
          < & \frac{2}{3}\e + \left| \sup \{\chi: \delta^N_{f,K,\chi}(\hat{\e})>0\} - \limsup_{n \to \infty} \frac{\log(N_{f,K}(\hat{\e},n))}{n}   \right|.
    \end{split}
\end{equation*}

By \Cref{eq:ExponentialEquivalence},

\begin{equation*}
    \left| h_{\topo,a_{\chi}} - h_{\topo} \right| < \frac{2}{3}\e.
\end{equation*}

To prove that $h_{\mu,a_{\chi}}(f)=h_{\mu}(f)$. First, assume that $h_{\mu}(f)$ is finite. Let  $\mathcal{P}=\{P_1,\dots,P_k\}$ be a finite measurable partition with $h_\mu(f,\P)=h$. By Shannon-McMillan-Breiman Theorem \cite[Theorem 9.3.1]{Viana_Oliveira_2016}, for the partition $\mathcal{P}$ we have that
\begin{equation*}
-\frac{1}{n} \log\mu([x]_0^{n-1})\rightarrow{h} \quad \mu{\text -a.e.}
\end{equation*}
 We use $[x]_0^{n-1}$ to denote the atom in $\displaystyle\vee_{i=0}^{n-1}T^{-i}\P$ that contains $x$. Let  $\e>0$, for $n$ be large enough, we get 
\begin{equation*}
    \mu\Big(x:\left|-\frac{1}{n} \log\mu([x]_0^{n-1})-h\right|<\epsilon\Big)>1-\epsilon.
\end{equation*}

This implies: 
\begin{equation}
\label{equation exp}
    \mu\Big(x:e^{-n(h+\epsilon)}<\mu([x]_0^{n-1})<e^{-n(h-\epsilon)}\Big)>1-\epsilon.
\end{equation}
Note for any $y\in [x]_0^{n-1}$, 
$y_j=x_j$ for any $0\leq j \leq n-1$,
therefore $S_{f,\P}(\e,n)\leq e^{n(h+\epsilon)}$.
Hence, when $a_{\chi}(n)=e^{\chi n}$, 
$h_{\mu,a_{\chi},\P}(f)\leq 
h$. Moreover, $B_{f,\P}^{n} (x,\e)$
is covered by at most $\dbinom{n}{\lfloor n\e \rfloor}(\lfloor n\epsilon\rfloor^k)$ 
atoms in $\displaystyle\bigvee_{i=0}^{n-1}T^{-i}\P$. 
This is because for any $y\in  B^{n}_{f,\P}(x,\e)$, $\{x_i\}_{i=0}^{n-1}$ and $\{y_i\}_{i=0}^{n-1}$ differ in at most $\lfloor n\epsilon\rfloor$ positions, each one
has at most $k$ choices. From \Cref{equation exp}, each atom is of measure at least $e^{-n(h+\e)}$, and they cover space of measure at least $1-\e$.
Hence, 
$S_{f,\P}(\e,n)\geq (1-\epsilon)e^{n(h-\mathcal{O}(\epsilon))}$ since $k$ is a fixed constant  
and 
\begin{equation*}
    \dbinom{n}{n\epsilon}\approx e^{n(\epsilon\ln(\epsilon)+(1-\epsilon)\log(1-\epsilon))}/\sqrt{2\pi n \e (1-\epsilon)}
\end{equation*}
when $n$ is large enough. Therefore, $h_{\mu,a_{\chi},\P}(f)\geq h$, proving the equality $h_{\mu,\P}(f) = h_{\mu,a_{\chi},\P}(f)$. By definition, $\displaystyle\sup_{\P}h_{\mu,\P}(f)=h_{\mu}(f)$ and $\displaystyle h_{\mu,a_\chi}(f) = \sup_{\P}h_{\mu,a_\chi,\P}(f)$, these give $h_{\mu}(f)=h_{\mu,a_\chi}(f)$.
If $h_\mu(f)=\infty$, then for any finite partition $\P$ with finite entropy, we still have $h_{\mu,\P}(f)=h_{\mu,a_{\chi},\P}(f)$, and the result follows by taking the $\sup$ over a sequence of partitions with finite entropy that go to infinity. Therefore, 
$h_{\mu}(f)=h_{\mu,a_\chi}(f)=\infty$.

We omit the proof of $h_{\semitop,\mu,a_\chi} = h_\mu$ and refer the reader to \cite[Theorem (I.I)]{KatokLyapunov}.
\end{proof}

\subsection{Shift complexity as slow entropy}
\label{sec:complexity}
Consider a finite alphabet $\mathcal A = \set{1,\dots,n}$ and the space 
\begin{equation*}
    \Omega = \mathcal A^\Z = \set{\omega = (\dots,\omega_{-2},\omega_{-1},\omega_0,\omega_1,\omega_2,\dots) : \omega_i \in \mathcal A \mbox{ for all }i \in \Z}.
\end{equation*}

The set $\Omega$ is called the {\bf shift space on $n$ symbols} and has a canonical dynamical system attached, the shift map $\sigma : \Omega \to \Omega$ defined by
\begin{equation*}
    \sigma(\omega)_n := \omega_{n+1}.
\end{equation*}

In other words, the sequence $\sigma(\omega)$ is the same as $\omega$, except that the 0 position of the sequence is shifted to the right by one index. A {\bf subshift} is a closed $\sigma$-invariant set $X$, and the {\bf language} of $X$ is the set 
\begin{equation*}
    \mathcal L = \set{ (\alpha_0,\dots,\alpha_m) : \mbox{there exists }\omega \in X \mbox{ with }\omega_i = \alpha_i \mbox{ for all }i=0,\dots,m}.
\end{equation*}

That is, $\mathcal L$ contains all of the finite words in $X$. To clarify the dynamical system, we let $\sigma_X$ denote the restriction of $\sigma$ to $X$. One may consider the {\bf (language) complexity} of $X$, which counts the growth rate of $\mathcal L$. That is, if we let $\mathcal{L}_m$ denote the words in the langauge of length $m$, we consider the function
\begin{equation*}
    p_n(X) = \# \mathcal{L}_m.
\end{equation*}

The language complexity has been studied carefully for a variety of subshifts, and we will not provide an exhaustive survey here, but some recent works on the complexity of subshifts and their applications include \cite{CK20,CJKS22,DOP22,CP23,PS23}. 

Fix the following metric on $\Omega$ (and correspondingly the induced metric on $X$) as
\begin{equation*}
d(\omega,\eta) = 2^{-k}, \mbox{ where } k = \inf\set{i \ge 0 : \omega_i \not= \eta_i \mbox{ or }\omega_{-i} \not= \eta_{-i}}.    
\end{equation*}

We can use the complexity to compute the topological slow entropy. Recall the definition of $N_{f,\Omega}(\ve,n)$ as given at the start of  \Cref{sec:top-slow}

\begin{PROP}
\label{prop:complexity}
    If $\sigma_X : X \to X$ is a subshift, then
    \begin{equation*}
        N_{f,\Omega}( 2^{-(k-1)},n) = p_{2k+1+n}(X).
    \end{equation*}
\end{PROP}

\begin{proof}
    Observe that by definition of the metric, $d(\omega,\eta) < 2^{-(k-1)}$ if and only if they agree on the indices ranging from $-k$ up to $k$. Hence $d(\sigma^\ell(\omega),\sigma^\ell(\eta)) < 2^{-(k-1)}$ if and only if $\omega$ and $\eta$ agree on the indices ranging from $\ell - k$ to $\ell + k$. Therefore, $\sigma$ and $\eta$ are $2^{-(k-1)}$-close in the metric $d^n_\sigma$ if and only if they agree on the indices ranging from $-k$ to $n+k$. Since there are $2k+1+n$ such indices, the follows that we for each finite word of length $2k+1+n$, we must choose a representative to cover the corresponsing Bowen ball: every such word has an element which must belong to a Bowen cover, and each Bowen ball must be centered at some point and hence can only cover one such word. The result follows.
\end{proof}

This yields the immediate corollary, which shows that in shift spaces, the topological slow entropy captures the growth rate of the complexity function.

\begin{COR}
\label{cor:TopEntropy&Compl}
    If $a_\chi$ is a scale, $\sigma_X : X \to X$ is a subshift, and
    \begin{equation*}
        \limsup_{n \to \infty}\dfrac{p_n(X)}{a_\chi(n)} = \left\lbrace\begin{array}{ll}\infty, & \chi < \chi_0 \\ 0, & \chi > \chi_0\end{array}\right.
    \end{equation*}
    then $h_{\topo,a_\chi}(\sigma_X) = \chi_0$.
\end{COR}

\section{Structural Theorems}
\label{sec:structuralthms}

\subsection{Slow Goodwyn's theorem}
The following result states a relation among the different entropies that we defined in \Cref{sec:defs}.  In the setting of classical entropy theory, this is part of the variational principle, due to Goodwyn \cite{goodwyn}. It states that under general circumstances, the metric entropy is bounded above by the semi-topological entropy and that the topological entropy is the largest of the previous two.

\begin{THM}[Goodwyn] 
\label{thm:SlowVariationalPrinciple}
    Let $f \colon X \to X$ be a measurable transformation of a compact metric space, and $\mathcal M^f(X)$ denote the space of $f$-invariant Borel probability measures on $X$. Then for any measure $\mu \in \mathcal M^f(X)$
\begin{equation}
    h_{\mu,a_\chi}(f) \le h_{\semitop,\mu,a\chi}(f) \le h_{\operatorname{top},a_\chi}(f).
    \label{eq:SlowGoodwyn}
\end{equation}

\end{THM}

\begin{RMK}
    As in \Cref{rmk:discontinuities}, we note that continuity is not required for this theorem.
\end{RMK}

To prove this theorem, we need an auxiliary result on the existence of certain partitions. In this result, the remarkable part is that the atoms of the partitions are small and their boundary is of measure zero.

\begin{LEM}
    Let $(X,d)$ be a compact metric space with a probability measure $\mu$. Then for every $\delta > 0$, there exists a partition $\P = \{P_1,\dots,P_k\}$ such that $\operatorname{diam}(P_i) < \delta$ for every $i = 1,\dots,k$ and if $X_\e = \{ x \in X : d(x,\partial \P) < \epsilon\}$, $\lim_{\e \to 0} \mu(X_\epsilon) = 0$, where 
    \begin{align*}
        d(x,\partial \P)\dfn \max _{P \in \P} d (x,\partial P).
    \end{align*}
    Furthermore, if $x,y \in X$ belong to different atoms of the partition and $x,y \not\in X_\epsilon$, then $d(x,y) \ge \epsilon$.
    \label{lemma:SpecialPartition}
\end{LEM}

\begin{proof}[Proof of \Cref{thm:SlowVariationalPrinciple}]
    We will start by proving that $h_{\mu,a_\chi} \le h_{\semitop,\mu,a_\chi}$. Let $\e>0$. 
    Let $\P_\e$ be the partition in \Cref{lemma:SpecialPartition} for $\delta = \e$, so each set in $\P_\e$ has diameter at most $\e$. From \Cref{lemma:SpecialPartition}, we conclude that the open set $X_{\hat{\e}}$ has $\mu$ measure at most $\e$ for some $\hat{\e}<\e$, and if $x,y \in X_{\hat{\e}}$, $d(x,y)<\hat{\e}$, then $[x]_{\P_\e} = [y]_{\P_\e}$. 
    We have that the compact set $K_{\hat{\e}}\dfn X \backslash X_{\hat{\e}}$ has $\mu$ measure at least $1-\e$, and satisfies that for any $x \in K_{\hat{\e}}$, we must have
    \begin{equation}
        B^n_{f,\P_\e}(x,\e) \supset B^n_{f}(x,\hat{\e}).
        \label{eq:HammingBowen}
    \end{equation} 
    By definition of $S_{f,\P_\e}(\e,n)$, for every $n$, we may choose $F_n \subset K_{\e}$, such that $\card(F_n) = S_{f,\P_\e}(\e,n)$ and 
    \begin{equation*}
        \mu\left( \cup_{x \in F_n} B_{f,\P_\e}^n(x,\e)\right)>1-\e.
    \end{equation*}
    Using \Cref{eq:HammingBowen}, we see that
    \begin{equation*}
        S_{f,\P_\e}(\e,n) \leq S_{\semitop,f}(\hat{\e},n) = \min\{\card(H):
        \mu\left( \cup_{x \in H} B^n_{f}(x,\hat{\e}) \right)>1-\hat{\e} \}, 
    \end{equation*}
    and 
    \begin{equation*}
        \delta_{f,\P,\chi}  ^S (\e) \leq \delta_{\semitop,f,\chi} ^S (\hat{\e}).
    \end{equation*}
    Since $\e \to 0$ implies that $\hat{\e} \to 0$, it follows that
    \begin{equation}
        \lim_{\e\to 0}\left(\sup_\chi\left\{\delta_{f,\P,\chi}  ^S (\e)>0\right\} \right)
        \le 
        \lim_{\hat{\e}\to 0}\left(\sup_\chi\left\{\delta_{\semitop,f,\chi} ^S (\hat{\e})>0\right\} \right).
    \label{eq:AlmostMeasureSmallerThanSemi}
    \end{equation}
    What we have accomplished with the \Cref{{eq:AlmostMeasureSmallerThanSemi}} is that $h_{\mu,a_\chi,\P_\delta} \le h_{\semitop,\mu,a_\chi},$ for a generating partition $\P_\delta$ with atoms of diameter at most $\delta>0.$
     The result follows from  \cite[Proposition 1]{KatokThouvenot} where the authors proved that the $\sup$ in $h_{\mu,a_\chi}=\sup_\P h_{\mu,a_\chi,\P}$ can be replaced by a limit $\lim_{m\to\infty} h_{\mu,a_\chi,\P_m}$ over a sequence of generating partitions $\{\P_m\}$.

     Now we prove the inequality $h_{\semitop,\mu,a_\chi} \le h_{\topo,a_\chi}$, this part does not require continuity and follows directly from definitions.  Notice that in \Cref{def:MetricEntropy}, we can drop the $\sup$ over compact subsets of $X$ and substitute $K=X$. Observe that for all $\e>0$, and all $n$ sufficiently large:
     \begin{equation*}
         S_{\semitop,f}(\e,n) 
         \leq 
         N_{f,\chi}(\e,n).
     \end{equation*}
     This completes the proof of the inequalities in \Cref{eq:SlowGoodwyn}.
\end{proof}

\begin{proof}[Proof of \Cref{lemma:SpecialPartition}]
Since $\mu$ is finite, the set $X'$ of atoms of $\mu$ is at most countable. 
For $\delta>0$, let $0<\delta'< \delta/2$ be 
such that if $B$ is a ball with center in an atom and radius $\delta'$, then $\mu(\partial B) = 0$. Similarly,
compact set $X \backslash \cup_{x \in X'} B(x,\delta')$ has a finite covering of balls of radius at most 
$\delta/2$, and the boundary of these balls has measure zero. 
Hence, there exists a finite covering of balls $B_1, \dots, B_k$ with radius at most $\delta/2$ and 
$\mu(\partial B_j) = 0$, for $1 \le j \le k$. The partition $\P$ is constructed by recursion: We put 
$P_1 = \bar{B_1}$, and $P_j = \bar{B_j} \backslash \cup_{i=1}^{j-1}\bar{B_i}$ for $2 \leq j\leq k$.

It follows that $\lim_{\e \to 0}\mu(X_\e) =\mu(\cup_{i=1}^k \partial P_i) =0$, 
because the probability measure is outer regular and $\cup_{i=1}^k \partial P_i \subset \cup_{i=1}^k \partial B_i.$

To prove the final observation, let $x,y \not \in X_\e$, i.e. $d(x,\cup_{i=1}^k \partial P_i) >\e$ and
$d(y,\cup_{i=1}^k \partial P_i)>\e.$ If there are $1 \le i < j \le k$, with $x \in P_i$, and $y \in P_j$, 
then $d(x,y)> d (x, \partial B_i) + d(y,\partial B_i) \geq 2\e.$
\end{proof}

\subsection{Homogeneous Measures and Slow Entropy}

The following result is of a classification type. It states that for a class of systems, both the semi-topological and topological entropies are the same. In other words, the inequality on the right in \Cref{eq:SlowGoodwyn} is in fact, an equality.

\begin{DEF}
\label{def:homog}
Let $f \colon X \to X$ be a measurable transformation. A measure $\mu$ is called {\bf homogeneous} with respect to $f$, if for every $\e > 0$ there exists $c > 0$ such that for any $x,y \in X$, and every $n \geq n_0$, for some $n_0$ sufficiently large, 
\begin{equation*}
    \frac{1}{c} \le \frac{\mu(B_f^n(x,\e))}{\mu(B_f^n(y,\e))} < c.
\end{equation*}
\end{DEF}

\begin{THM}
\label{thm:homogmeasures}
    Suppose that $f \colon X \to X$ is a measurable transformation, $X$ is compact, and $\mu$ is homogeneous. Then 
    \begin{equation*}
         h_{\topo,a_\chi}(f) = h_{\semitop,\mu,a_\chi}(f) = \lim_{\e \to 0} \left(\sup \{ \chi : C_{f,\chi}(\e) > 0 \}\right)
    \end{equation*}
    where
    \begin{equation*}
     C_{f,\chi}(\e) = \limsup_{n\to\infty}\dfrac{1}{a_\chi(n) \cdot \mu(B_f^n(x,\e))}.
    \end{equation*}
\end{THM}

\begin{proof}
Let $\displaystyle k_{\mu ,a_{\chi }} (f)$ be the quantity $\displaystyle\lim_{\e \to 0} \left(\sup \{ \chi : C_{f,\chi}(\e) > 0 \}\right).$
  
  First we show that $h_{\topo,a_{\chi }} (f) = k_{\mu ,a_{\chi }}(f). $ Given a compact set $K$, by locally compactness of $X$, we can choose a sufficiently small $\e $ and an open set $U$  such that $B (K,\e ) \subset U.$ Denote $ \hat{\chi }:=\sup \left\{ \chi : \delta ^S_{f,K, \chi } (\e ) >0\right\} $. 

  On the one hand, fix a set $E$ such that for all distinct $x,y \in E$ we have $B^n_f(x,\e ) \cap B^n _f (y,\e ) = \emptyset,$  and thus $\cup _{y \in E} B^n_f(y,\e ) \subset U$ is a disjoint union. Since $\mu $ is homogeneous, we can choose $c>0$ such that for any $x,y \in X,$  
  \begin{align*}
    \mu \left( B^n_f (x,\e ) \right) \le c \mu \left( B^n_f (y,\e ) \right) .
  \end{align*}
  Summing over $y \in E$ we obtain 
  \begin{align*}
    \mu \left( B^n_f (x,\e ) \right) \card (E) \le c \sum_{y \in E} \mu \left( B^n_f(y,\e ) \right) \le c \mu (U) \le c,
  \end{align*}
  which implies that for any $x \in X$ we have
  \begin{align*}
    \mu \left( B^n_f (x,\e ) \right) \cdot S_{f,K} (\e,n)  \le c,
  \end{align*}
  therefore 
  \begin{equation*}
      0< \frac{1}{c} \le \limsup_{n\to \infty} \dfrac{1}{ \mu(B_f^n(x,\e))\cdot S_{f,K}(\e,n)}.
  \end{equation*}
  We will use the following, although we do not prove it. This follows from the definitions of $\lim$ and $\limsup$
  \begin{CLAIME}
      For two sequences $\{a_n\}$ and $\{b_n\}$, if $\lim_{n \to \infty} a_n$ exists and it is positive and $\limsup_{n\to \infty} b_n$ is positive, then
      \begin{equation*}
          \limsup_{n \to \infty} a_n \cdot b_n = \lim_{n\to \infty}a_n \limsup_{n \to \infty}b_n.
      \end{equation*}
  \end{CLAIME}
  For $\chi <  \hat{\chi }, $ take $\chi<\chi'<\hat{\chi}$, using the claim, we have that
  \begin{align*}
    C_{f ,\chi }(\e) 
    &= \limsup_{n\to\infty}\dfrac{1}{a_\chi(n) \cdot \mu(B_f^n(x,\e))} 
    \\ 
    &=\limsup_{n\to\infty}\dfrac{S_{f,K}(\e,n)}{a_{ \chi'} (n) }\cdot \dfrac{a_{ \chi' }(n)}{a_{\chi}(n)} \cdot\dfrac{1}{ \mu(B_f^n(x,\e)) \cdot S_{f,K}(\e,n)}  
    \\
    &= \infty .
  \end{align*} 
  Denote $k_{\mu ,a_{\chi },\e} (f):= \sup \{ \chi : C_{f,\chi}(\e) > 0 \}$ and 
  \begin{equation*}
      h_{\topo,a_\chi,K,\e}(f):=
     \sup \left\{ \chi:
    \delta_{f,K,\chi}^S(\e)>0 \right\}.
  \end{equation*}
It follows that for any $\chi <  \hat{\chi } $ we have 
  \begin{align*}
    k_{\mu , a_\chi, \e  } (f ) \ge \chi .
  \end{align*}
  By taking supremum over $\chi <  \hat{\chi } $ on both sides we obtain 
  \begin{align*}
    k_{\mu ,a_{\chi }, \e } (f ) \ge  \hat{\chi }  = h_{\topo,a_{\chi }, K , \e } (f).
  \end{align*}
  Since both $K$ and $\e $ are arbitrary, it is clear that 
  \begin{align*}
    k_{\mu ,a_{\chi }} (f) \ge h_{\topo, a_{\chi }} (f).
  \end{align*} 

  On the other hand, assume $\mu (K) >0$.  Since  $\mu$ is homogeneous, we have 
  \begin{align}
    \label{eq: k=top covering}
    \mu \left( B^n_f(x,\e ) \right) \ge \frac{1}{c}  \mu \left( B^n_f(y,\e ) \right).    
  \end{align}
  Fix a finite covering $\left\{ B^n_f(x,\e ) : x \in F  \right\}$ of $K$ consisting of $(n,\e )$-Bowen balls so that $\cup_{y \in F} B^n_f(y,\e ) \supset K$.  By summing \Cref{eq: k=top covering} over $y \in F$,  we obtain
  \begin{align*}
    \mu \left( B^n_f(x,\e ) \right) \card (F) \ge \frac{1}{c} \cdot \sum_{y \in F} \mu \left( B^n_f(y,\e ) \right) \ge \frac{1}{c} \mu (K),
  \end{align*}
  which implies that for any $x \in X$ we have 
  \begin{align*}
    \mu \left( B^n_f(x,\e ) \right) N_{f,K} (\e,n) \ge  \frac{1}{c}  \mu (K).
  \end{align*}
  For $\chi >  \hat{\chi } $ we have 
  \begin{align*}
    C_{f ,\chi }(\e) 
    &= \limsup_{n\to\infty}\dfrac{1}{a_\chi(n) \cdot \mu(B_f^n(x,\e))} 
    \\  
    &=\limsup_{n\to\infty}\dfrac{N_{f,K}(\e,n)}{a_{ \hat{\chi } (n) }}\cdot \dfrac{a_{ \hat{\chi } }(n)}{a_{\chi } (n)} \cdot\dfrac{1}{ \mu(B_f^n(x,\e)) N_{f,K}(\e,n)}  
    \\
    &= 0,
  \end{align*} 
  and thus $k_{\mu ,a_{\chi },\e } (f) \le  \chi $ for any $\chi >  \hat{\chi } .$ It follows that $k_{\mu ,a_{\chi },\e } (f) \le  \hat{\chi } = h_{\topo,a_{\chi }, K , \e } (f) $ and thus $k_{\mu ,a_{\chi }} (f) \le h_{\topo, a_{\chi }} (f).$ Now we can conclude that $h_{\topo, a_{\chi }} (f) = k_{\mu ,a_{\chi }} (f).$  

Secondly, we show that $h_{\semitop, a_{\chi }} (f) \ge k_{\mu , a_{\chi }} (f).$ Given $\e >0,$  define 
  \begin{align*}
    \tilde{\chi }
    &\dfn \sup \left\{ \chi  : \delta_{\semitop,f,\chi} ^S (\e)  >0  \right\},
  \end{align*} 
  in which case we have $\displaystyle  \delta_{\semitop,f,\tilde{\chi }}^S (\e)  = \limsup _{n \to \infty } \frac{S_{\semitop,f} (\e , n)}{a_{\tilde{\chi }}(n)},$ where 
  \begin{align*}
    S_{\semitop,f} (\e ,n) = \min \left\{ \card (F) : \mu \left( \cup_{x \in F} B^n_f (x,\e ) \right) > 1- \e   \right\}.
  \end{align*} 
  Consider a finite measurable set $F$ such that $\mu \left( \cup_{x \in F} B^n_f (x,\e ) \right) > 1- \e .$
  Since $\mu $ is homogeneous, then there exists $c>0$  such that $\mu \left( B^n_f (x,\e ) \right) \ge \frac{1}{c} \mu \left( B^n_f (y,\e ) \right). $    Summing over $y \in F $ we obtain 
  \begin{align*}
    \mu \left( B^n_f (x,\e ) \right) \card (F)
    \ge \sum_{y \in F} \frac{1}{c} \mu \left( B^n_f (y,\e ) \right) 
    \ge \frac{1}{c} \mu \left( \cup_{y \in F}  B^n_f (y,\e )\right)  
    > \frac{1-\e }{c}.
  \end{align*} 
  Since $F$ is arbitrary, we have $\mu \left( B^n_f (x,\e ) \right) S_{\semitop,f} (\e ,n) \ge \frac{1-\e }{c} .$ Hence for $\chi > \tilde{\chi }$ 
  \begin{align*}
    C_{f, \chi} (\e ) 
    &= \limsup _{n \to \infty } \frac{1}{a_{\chi }(n) \cdot \mu \left( B^n_f (x, \e ) \right) }
    \\ 
    &= \limsup _{n \to \infty } \frac{a_{\tilde{\chi }} (n)}{a_{\chi }(n)} \cdot \frac{S_{\semitop,f} (\e , n)}{a_{\tilde{\chi }}(n)} \cdot  \frac{1}{ \mu \left( B^n_f (x, \e ) \right) \cdot S_{\semitop,f} (\e , n) } 
    \\ 
    & \le \frac{c}{1-\e }  \lim_{n \to \infty} \frac{a_{\tilde{\chi }} (n)}{a_{\chi }(n)} \limsup_{n \to \infty } \frac{S_{\semitop,f} (\e , n)}{a_{\tilde{\chi }}(n)} = 0.
  \end{align*}
  Letting $\e \to 0$, as in the first part, we obtain that  
  $k_{\mu ,a_{\chi }} (f) \le h_{\semitop, a_{\chi }} (f).$ 
   
  Finally, 
  by \Cref{thm:SlowVariationalPrinciple} 
  we have $h_{\semitop, a_{\chi}} (f) \le h_{\topo, a_{\chi}}$, we proved that $h_{\topo,a_{\chi}} (f) = k_{\mu, a_{\chi}} (f)$ in the first part. Thus the \Cref{thm:homogmeasures} follows.
\end{proof}

\section{Ferenczi's Theorem}
\label{sec:ferenczi}

Here we will present the proof of \cite[Proposition 3]{Ferenczi0}, which provides a characterization of Kronecker systems via slow entropy. This section is purely expository, we include it to provide a more complete account of the current state of the theory.

\begin{DEF}
A topological dynamical system is called a {\bf Kronecker system} if it is isomorphic to a group rotation on a compact abelian metrizable group. 
\end{DEF}

\begin{THM}[\cite{Ferenczi0}]
 Let $\bar{X}=(X,\mu,f,\mathcal{M})$ be a probability measure preserving system. Then, $\bar{X}$ is isomorphic to the Kronecker system if and only if $h_{\mu,a_\chi}(f)=0$ for all scales $a_\chi.$
 \label{thm:Ferenczi}
\end{THM}

For the second part of the proof of \Cref{thm:Ferenczi}, we need the following result.

\begin{LEM}
\label{almost periodic function}
If $\bar{X}=(X,\mu,f,\mathcal{M})$ is not isomorphic to a Kronecker system, then we can find a partition $\P=\{P_1,P_2...,P_{l}\}$ and $\e_0>0$ such that, $d(f^n(\P),\P)>\e_0$ for every $n$ in a density 1 subset $D \subseteq \mathbb{N}$. 
\end{LEM}

For two partitions $\P=\{P_1,\dots,P_n\}$ and $\mathcal{Q}= \{Q_1,\dots,Q_n\}$, we denote the {partition distance}, between $\P$ and $\mathcal{Q}$  by 
\begin{align*}
    d(\P,\mathcal{Q})\dfn\displaystyle\sum_{i=1}^n \mu(P_i\Delta Q_i).
\end{align*}

\begin{proof}[Proof of \Cref{thm:Ferenczi}]
 If $\bar{X}$ is isomorphic to the Kronecker system, we can assume $f$ is an isometry and $X$ is a compact metric space with metric $d$. Let $\delta >0$ be a fixed constant, and let $\P \dfn \P_\delta$ be a measurable partition of $X$ described in  \Cref{lemma:SpecialPartition}. Given $\e>0$, we want to show $S_{f,\P}(\e,n)\leq C_\epsilon$ when $n$ is large enough where $C_\epsilon$ is a constant independent of $n$. Let \(\delta(\e)\) be a constant depending only on \(\e\), define $X_\epsilon=\{x\in X: d(x,\partial{P})<\delta(\e) \text{\;for some \;} P\in \P\}$. By taking \(\e\) small enough, $\mu(X_\epsilon)<\e$. Let $C_\epsilon$ be the minimal number of balls of radius
 $\delta(\e)/2$ covering $X$. Now, for every $n$ large enough
 \begin{equation}
     \mu\left(\left\{ x: \left|\frac{1}{n} \sum_{i=0}^{n-1}\mathds{1}_{X_{\epsilon}}(f^i(x))\right|<\epsilon\right\}\right)>1-\epsilon.
     \label{eq:NOfBoundary}
 \end{equation}
 Therefore, by \Cref{eq:NOfBoundary}, the set of indices 
 \begin{equation*}
 E_x^n=\left\{0\leq j <n: f^j(x)\in X_{\epsilon}^c \right\}
 \end{equation*}
 satisfies that $|E_x^n| > n(1-\epsilon)$
 for $x$ in a set with measure at least $1-\epsilon$. For such $x$, suppose $y$ is in the same ball of radius $\delta(\e)/2$ containing $x$. Then, when $j\in E_x^n$, $f^j x$ and $f^jy$ are in the same atom of partition $\P$ because $f$ is an isometry.  Therefore, $S_{f,\P}(\e,n)\leq C_\epsilon$ because each Hamming ball contains a Bowen ball of radius $\delta(\e)/2$. Since the Bowen balls are exactly the balls in the metric $d$, $S_{f,\P}(\e,n)$ is bounded by a constant depending only on $\mathcal P$ and $\e$. Consequently, any Kronecker system has zero entropy at all scales.
 
 To prove the other direction, we need \Cref{almost periodic function}.
 Assume $h_{\mu,a_\chi}(f)=0$ for all scales $a_\chi$, we claim that \[\displaystyle\liminf_{n\rightarrow \infty}S_{f,\P}(\e,n)\leq C_\epsilon\] for any fixed measurable partition $\P$ and given $\e$. Suppose not, then there exists a measurable partition $\P_0$ and $\e_0$ $s.t.$ for any $k$ there exists $n_k$ satisfying 
$S_{f,\P_0}(\e,n)>k$ for every $\e<\e_0$ and $n_k<n$. Choose $a_{\chi}(n)$ to be some value between $k$ and $k+1$ when $n_k\leq n< n_{k+1}$. This implies $h_{\mu,a_{\chi},\P_0}(f)>0$, which is a contradiction. 

If the system $\bar{X}$ is not isomorphic to a Kronecker system, by 
applying \Cref{almost periodic function} we can find a measurable partition $\P=\{P_1,...,P_l\}$. Fixed an arbitrarily large $N$, there exists a constant $C$ $s.t.$ 
$S_{f,\P}(\e^2,N)\leq C$. By using  $M\leq C$ Hamming balls     $B_i=B_{f,\P}^N(x^i,\e^2)$ where $1\leq i\leq M$, one can cover at least $1-\e^2$ space of $X$. Consider space $(X\times \{0,1,\dots,N-1\}, \mu \times \nu )$, where $\nu(E)= \frac{1}{N} \times \text{cardinality of } E$. Define a measurable function on $X\times \{0,1,\dots,N-1\}$ by
$$ g(x,n)=\left\{
\begin{aligned}
1 & \; \text{if there exists $i$ $s.t.$  $x \in B_i$ and  $f^n(x)$ and $f^n(x^i)$ are }\\
  &\; \text{in the same atom}\\
0 &\; \text{others} \\
\end{aligned}
\right.
$$
By definition of Hamming balls, \[\int g \, d\mu \times \nu > (1-\epsilon^2)(1-\epsilon^2).\]
Now, there exists a subset $E$ of $\{0,1,\dots,N-1\}$ with cardinality larger than $N(1-4\epsilon)$ 
satisfying for every $n \in E$, $\int g \, d\mu>1-\epsilon/2$. 
This is because we can get  $\int g d\nu < 1-2\e^2$ otherwise.

We claim that we can find a positive density subset $\Lambda$ of $\mathbb{N}$ $s.t.$ for every $n\in \Lambda$, $d(f^{n}(\P),\P)<\e$, which is a contradiction to \Cref{almost periodic function}.
Note that \[d(f^{n}(\P),f^{j}(\P))=\int_{X} \min\{|x_n-x_j|,1 \}d\mu(x)\] for every $n$ and $j$. Moreover, for each 
$n$, $ (x_n^1,\dots,x_n^M )$ has at most $l^M$  choices. Therefore, for every $n\in E$, there exists an integer
$m(n)$, $s.t.$ $x_n^i=x_{m(n)}^i$ for every $1\leq i\leq M$. Hence, $d(f^n(\P),f^{m(n)}(\P))<\e$
because the set satisfying there exists $i$ $s.t.$ $x\in B_i$ and $x_n^i=x_n$ has measure larger
than $1-\e$ for every $n\in E$. From the pigeonhole principle, there exists a subset $\Lambda_N$ of 
$\{0,1,\dots,N-1\}$ with cardinality at least $N(1-4\e)/l^M$ such that for every $n\in \Lambda_N$, 
$d(f^n(\P),f^{m}(\P))<\e$ for some $m\in \{0,1,\dots,N-1\}$. Since the metric is $f$-invariant, we can conclude the result.
\end{proof}

\begin{DEF}
If $\bar{X}=(X,\mu,f,\mathcal{M})$ is a measure preserving system, we say a function $g \in L^2(X,\mu)$ is {\bf almost periodic} if $\{U_{f}^n g : n\in \mathbb{Z}\}$ is a precompact set in $L^2(X,\mu)$.
\end{DEF}

\begin{proof}[Proof of \Cref{almost periodic function}]
    We give a sketch of proof here.
    If the system is isomorphic to a Kronecker system, then $g$ is almost periodic for any $g\in L^2(X)$.
    Suppose $\bar{X}$ is not isomorphic to the Kronecker system, then there exists a non-zero measurable function $g$ with mean $0$ such that $g$ is orthogonal to every almost periodic function. Then, there is a density 1 subset $D$ of $\mathbb{N}$ $s.t.$ \[\displaystyle\lim_{n\rightarrow \infty, n\in D}\langle U_f^n(g),g \rangle =0.\] Assume w.l.o.g. that \(\|g\|=1\), given \(\epsilon=\frac{1}{100},\) there is a simple function \(h=\sum\limits_{i=1}^{l}c_i\mathds{1}_{A_i}\) such that \(\|g-h\|\leq \e/4.\) Therefore, when \(n\) is large enough, \[|\langle U_f^n(h),h\rangle|\leq \e.\]  Choose partition \(\P=\{A_1,...,A_l\}\), then there exists \(\e_0<\frac{\min\{\mu(A_1),...,\mu(A_l)\}}{100} \), such that \(d(f^n(\P),\P)>\e_0\) for \(n\in D.\)
\end{proof}

\section{Differences between Slow and Exponential Entropy}

\subsection{Sturmian subshifts}
\label{sec:sturmian}
For more information about Sturmian systems, we refer the reader to \cite[Chapter 6]{Pytheas}.
\begin{DEF}
    A sequence $u \in \{0,1\}^{\N}$ $(\in \{0,1\}^\Z)$  is a (bi-infinite) {\bf Sturmian sequence} if for every $n \ge 1$, the number of words of length $n$ that appear in $u$ is equal to
    \begin{equation*}
        p_n(u) = n+1,
    \end{equation*}
    and it is not eventually periodic.
\end{DEF}
Consider $f \colon \{0,1\}^\Z \to \{0,1\}^\Z$, the shift map defined as $w = f(v)$, $w_i = v_{i+1}$.

\begin{DEF}
    A {\bf Sturmian system} $(X_u,f)$ for a bi-infinite Sturmian sequence $u$ is the shift map $f$ on  
    $X_u = \overline{\{f^k(u):k\in \Z\}}$.
\end{DEF}

\begin{THM}
    Any Sturmian system $(X_u,f)$ is measurably equivalent to an irrational rotation $(\R/\Z, R_{\theta})$.
    \label{thm:Equiv2Rot}
\end{THM}
\begin{proof}[{Sketch of the proof with ideas in \cite[Chapter 6]{Pytheas}}
]
   Any initial point $\beta \in \R/\Z$, generates a Sturmian sequence $v = (v_i)$, by taking $v_i = j$ if $R_\theta^i(\beta) \in I_j$, $I_0=[0,1-\theta)$ and $I_1=[1-\theta,1)$.

   The converse is very elaborate see \cite[Sections 6.3 and 6.4]{Pytheas}. The idea is that every $v \in X_u$ has a coding $(a_n,b_n)$. The coefficients $(a_n)$ are the partial quotients of the continued fraction expansion of $\theta$. The initial point $\beta$ is determined by $b_n$; these are the coefficients of its Ostrowski expansion. Things are very subtle; for instance, see \cite[Exercise 6.2.13 item 5]{Pytheas}. 
\end{proof}

The \Cref{thm:Equiv2Rot} combined with \Cref{thm:Ferenczi} implies that Sturmian systems have 0 metric entropy at all scales.
The following immediately implies \Cref{thm:sturmian}:

\begin{THM}
\label{Sturmian}
    For any Sturmian system $(X_u,f)$, we have that 
    \begin{equation*}
        0=h_{\mu,a_\chi}<h_{\semitop,\mu,a_\chi}=h_{\topo,a_\chi}=1,
    \end{equation*}
    for the polynomial scale $a_\chi(N)=N^\chi$, and the unique $f$-invariant measure $\mu$.
\end{THM}

\begin{proof}


{\bf  Proof of $h_{\topo,a_\chi} = 1$.}



This follows from \Cref{prop:complexity} and \Cref{cor:TopEntropy&Compl}. For Sturmian shifts, we have that for a fixed integer $k > 0$, 
\begin{equation*}
    N_{f,X_u}(2^{-(k-1)},n) = p_{2k+n+1}(X_u) =2k+n+2.
\end{equation*}
Hence,
\begin{equation*}
\begin{split}
    \delta_{f,X_u,\chi}(2^{-(k-1)}) &= 
    \limsup_{n \to \infty} \frac{N_{f,X_u}(2^{-(k-1)},n)}{n^\chi}\\
    &= \lim_{n\to \infty}\frac{2k+n+2}{n^\chi}=\left\{\begin{array}{cc}
         0, & \chi >1, \\
         1, & \chi =1, \\
         \infty, & \chi <1.
    \end{array}\right.
\end{split}
\end{equation*}
For the linear scale $a_\chi(N) = N^\chi$,
\begin{equation*}
h_{\topo,a_\chi} = \lim_{\e}(\sup\{\chi:\delta_{f,X_u,\chi}(\e)>0\}) = 1.
\end{equation*}

{\bf  Proof of $h_{\semitop,\mu,a_\chi}=1$.}

Let $\mu$ be the invariant probability measure for $(X_u,f)$. If the sequence $u$ was coded by the partition $\P=\{P_0=[0,1-\theta),P_1=[1-\theta,1) \}$, then for every cylinder set $\omega \subset X_u$ defined by a word $b_0b_1 \dots b_{n-1} \in L_n$, we must have that $\mu(\omega) = \Leb(P)$ for some $P \in \P^n \dfn\vee_{i=0}^{n-1}R^{-i}_\theta\P. $ The atoms in the partition $\P^n$ have endpoints in 
\begin{equation*}
    \{0,R_{1-\theta}(0),\dots, R^{n}_{1-\theta}(0)\}.
\end{equation*}
Let $1-\theta = [0;a_1,a_2,a_3,\dots]$ be the continued fraction expansion, and $\{\frac{p_k}{q_k}\}$ the sequence of best approximants. Writing $n=mq_k +q_{k-1}+r$ with $1 \leq m \leq a_{k+1}$, and $0 \leq r < q_k$. By the {Three Gap Theorem}, see \cite[Section 3]{Alessandri2007ThreeDT}, the measure of the cylinder $\omega$ is 
\begin{enumerate}
    \item $\eta_k \dfn (-1)^k(q_k(1-\theta)-p_k)$, in which case there are $n+1-q_k$ cylinders of this measure. These are the smallest gaps.
    \item $\eta_{k-1} - m\eta_k$. There are  $r+1$ cylinders of this measure. 
    \item \label{item:bigcylinders} $\eta_{k-1} - (m-1)\eta_k$. There are $q_k-(r+1)$ cylinders of this measure. These are the biggest gaps, and their length is the sum of the lengths of the previous types.
\end{enumerate}
When $n=(m+1)q_k+q_{k-1}-1$, for $1\le m \le a_{k+1}$ ($r=q_{k}-1$), then there are no cylinders with the length in   \cref{item:bigcylinders}.

If 
\begin{equation*}
    C(n) \dfn \min\left\{\card(F):\sum_{\omega \in F}\mu(\omega) > 1-\e\right\}
\end{equation*}
where $F$ is a family of cylinders sets of size $n$, then for some $0\le l , s , \kappa \le 1 $ depending on $n$,
\begin{equation}
\label{eq:growthofbowenb}
    C(n) = n+1 - \left[\kappa(n+1-q_k)+s(r+1)+l(q_k-(r+1))\right]
\end{equation}
where 
\begin{equation}
\label{eq:restriction}
    \kappa(n+1-q_k)\eta_k + s(r+1)(\eta_{k-1}-m\eta_k) + l(q_k-(r+1))(\eta_{k-1}-(m-1)\eta_k) <\e.
\end{equation}
Specializing to the case where there are no cylinders  of type \cref{item:bigcylinders} 
\begin{equation*}
    n = n_k= (a_{k+1}+1)q_k+q_{k-1}-1 = q_{k+1}+q_k-1,
\end{equation*}
 we can substitute $l=0$ and $r=q_k-1$ into \Cref{eq:restriction} to obtain:
 \begin{equation}
 \label{eq:specialrestriction}
 \begin{split}
     -\kappa&>\frac{s(r+1)(\eta_{k-1}-a_{k+1}\eta_{k})
     -\e}{(n+1-q_k)\eta_k}
    \\
    &=\frac{sq_k\eta_{k+1}-\e}{q_{k+1}}.
 \end{split}
 \end{equation}
 
 In this special case, combining \Cref{eq:growthofbowenb} and \Cref{eq:specialrestriction}, we obtain that
\begin{equation}
\label{eq:longEstimateforGrowth}
\begin{split}
    \frac{C(n_k)}{n_k}&=
    1+\frac{1}{n_k}-
        \kappa\frac{q_{k+1}}{n_k}
        -
        s\frac{q_k}{n_k}
\\
& > 
    1 + \frac{1}{n_k}
    +
        \frac{(sq_k\eta_{k+1}-\e)}{q_{k+1}\eta_k}\frac{q_{k+1}}{n_k}
    -
        s\frac{q_k}{n_k}
\\
& = 1+\frac{1}{n_k}
+s\frac{q_k}{n_k}\left(\frac{\eta_{k+1}}{\eta_k}-1\right) - \frac{\e}{n_k\eta_k}.
\end{split}
\end{equation}
Using Khinchin's  inequality, see for instance \cite[Theorems 9 and 13]{Khinchin}:
\begin{equation*}
    \frac{1}{q_{k+1}+q_{k}}
    <
    \eta_k
    <
    \frac{1}{q_{k+1}}
\end{equation*}
and that $n_k= q_{k+1}+q_{k}-1$,
we obtain
\begin{equation}
\label{eq:limsupSteinhaus}
    1 = 1 - \lim_{k\to \infty}\frac{1}{q_{k+1}+q_{k}}
    \leq 
    \limsup_{k \to \infty} n_k\eta_k 
    \le 
    1 + \limsup_{k \to \infty}\frac{q_{k}}{q_{k+1}} < \infty.
\end{equation}
By \Cref{eq:longEstimateforGrowth} and \Cref{eq:limsupSteinhaus},
we have that 
\begin{equation*}
\delta_{\semitop,f,1} ^S(\e) =     \limsup_{n\to \infty} \frac{C(n)}{n} \geq \limsup_{k \to \infty}\frac{C(n_k)}{n_k}\geq 1-o(\e).
\end{equation*}
We conclude that $h_{\semitop,\mu,a_\chi} = 1$ by \Cref{eq:SlowGoodwyn} in \Cref{thm:SlowVariationalPrinciple}
\end{proof}

\subsection{Large gaps}

We thank Scott Schmieding for pointing out the constructions in this section. Fix the polynomial scale $p_\chi(N) = N^\chi$. The following shows that we can achieve gaps of polynomial size for the variational principle:

\begin{THM}
\label{thm:big-polynomial}
    For every $m \in \N$, there exists a uniquely ergodic homeomorphism of a compact metric space $f : X \to X$ preserving $\mu$ such that $h_{\mu,p_\chi}(f) = 0$, but $h_{\topo,p_\chi}(f) = m$.
\end{THM}

\begin{proof}
Consider finitely many irrational numbers $\beta_1,\dots,\beta_m \in (0,1)$ which are rationally independent. Then the translation on $\mathbb{T}^m$ by the vector $v = (\beta_1,\dots,\beta_m)$ is uniquely ergodic, and is isomorphic to the product of the rotations $R_{\beta_1}\times \dots \times R_{\beta_m}$.

For each $i = 1,\dots,m$, consider the Sturmian subshift $\sigma_i : X_i \to X_i$ isomorphic to the rotation $R_{\beta_i}$, and let $X = X_1\times \dots \times X_m$ and $f : X \to X$ be the product $f = \sigma_1 \times \dots \times \sigma_m$. Note that an element $x \in X$ is an $m$-tuple of infinite words in the symbols $0$ and $1$. Equivalently, one may consider it as a single infinite word whose entries are $m$-tuples of $0$'s and $1$'s. Therefore, $f$ may be considered a subshift of a shift on $2^m$ symbols. Since a word is admissible if and only if each of its components are admissible, and each component may be chosen independently from each corresponding Sturmian language, there are $(n+1)^m$ words of length $n$ in the language of $f$. Thus, the topological slow entropy at polynomial scale is $m$ by \Cref{cor:TopEntropy&Compl}.

On the other hand, we claim that $f$ is uniquely ergodic (in which case $f$ is measurably isomorphic to $R_{\beta_1}\times \dots \times R_{\beta_m}$). Indeed, given an $f$-invariant measure $\mu$, it must project to a $\sigma_i$-invariant measure on each $X_i$. Since Sturmian subshifts are uniquely ergodic, it follows that $\mu$ is a joining of the circle rotations $R_{\beta_i}$. Hence $\mu$ must correspond to the Haar measure.
\end{proof}

Fix the scale $a_\chi(N) = e^{N^\chi}$, which we call the {\bf stretched exponential scales}. Note that $a_\chi$ is faster than polynomial scales for all $\chi > 0$, but for $\chi < 1$, the rate $a_\chi$ is slower than exponential.

\begin{THM}
\label{thm:big-gap}
    There exists a uniquely ergodic subshift $\sigma : X \to X$ preserving a measure $\mu$ such that $h_{\mu,b_\chi}(\sigma) = 0$ for every family of scales $b_\chi$, but $h_{\topo,a_\chi}(\sigma) = 1$.
\end{THM}

\Cref{thm:big-gap} shows that there are systems that have very large gap between the metric and topological slow entropies, achieving stretched exponential rates arbitrarily close to 1.

\begin{RMK}
    \Cref{thm:big-gap} heavily relies on our use of $\limsup$ rather than $\liminf$ when defining our slow entropies. When using the $\liminf$ definition, the topological slow entropy grows linearly. That is, the growth rate of $N_{\sigma,X}(\ve,n)$ (which is linked to the language complexity by \Cref{prop:complexity}) oscillates between linear and stretched exponential rates.
\end{RMK}

\begin{proof}[Proof of \Cref{thm:big-gap}]
    In \cite[Theorem 5.15, Proposition 5.23]{PavlovSchmieding}, it is shown that among transitive subshifts $\sigma : X \to X$, the following properties (among others) are generic:

    \begin{itemize}
        \item $\sigma$ is a regular Toeplitz subshift
        \item For every $\gamma \in (0,1)$, $p_n := p_n(X)$ has subsequences satisfying

        \[ \lim_{n \to \infty} \dfrac{\log\log p_{n_k}}{\log n_k} = \gamma\]
    \end{itemize}

    Since such a $\sigma$ is a regular Toeplitz subshift, it is uniquely ergodic and measurably isomorphic to translation on a compact abelian group \cite{JacobsKeane}. It follows that the metric slow entropy is 0 at all scales. On the other hand, if $0 < \gamma < \gamma' < \gamma''$, identify a subsequence such that $\lim_{k \to \infty} \dfrac{\log\log p_{n_k}}{\log n_k} = \gamma''$. Then for sufficiently large $k$,

    \begin{eqnarray*}
        \dfrac{\log \log p_{n_k}}{\log n_k} &>& \gamma' \\
        \log \log p_{n_k} & > & \gamma' \log n_k \\
        \log p_{n_k} & > &n_k^{\gamma'} \\
        p_{n_k} & > & e^{n_k^{\gamma'}}.
    \end{eqnarray*}

    Since $\gamma$ was aribtrary, it follows that 
    \[ \limsup_{n\to \infty} \dfrac{p_n}{e^{n^\gamma}} = \infty\]
    whenever $\gamma \in (0,1)$. On the other hand, when $\gamma = 1$, the $\limsup$ must be 0 since the system has 0 topological entropy (since the variational principle holds at exponential scale, and the system is uniquely ergodic with 0 exponential metric entropy). It follows that the topological slow entropy at stretched exponential scale is 1 by \Cref{cor:TopEntropy&Compl}.
\end{proof}

\subsection{Denjoy circle transformations}
\label{sec:denjoy}
The Sturmian systems considered above can be realized as invariant sets for transformations of the circle. Indeed, one may build $C^{1,\alpha}$ circle diffeomorphisms by starting with an irrational circle rotation and ``blowing up'' an orbit by inserting an interval at each point of the orbit. Such examples were first studied by Denjoy and their construction can be found in \cite[Section 12.2]{hasselblatt-katok}. We characterize them here:

\begin{DEF}
    We say that a circle homeomorphism $f : S^1 \to S^1$ is {\bf Denjoy} if the rotation number $\theta$ of $f$ is irrational, and there is a semiconjugacy $h : S^1 \to S^1$ and a point $x_0 \in S^1$ such that
    
    \begin{itemize}
        \item $h \of f = R_\theta \of h$,
        \item $h^{-1}(f^n(x_0))$ is a nontrivial closed interval for all $n \in \Z$,
        \item $h^{-1}(x)$ is a single point for all $x$ outside of the orbit of $x_0$.
    \end{itemize} 
\end{DEF}

\begin{LEM}
    If $f$ is a Denjoy circle transformation, then $f|_{NW(f)}$ is topologically conjugated to a Sturmian subshift.
\end{LEM}

\begin{proof}[Sketch of proof]
    Recall that Sturmian sequences can be obtained by looking at codes appearing of the rotation $R_\theta$ using the intervals $[0,1-\theta)$ and $[1-\theta,1)$. In the case of a circle rotation the map which sends the code to the point is not one-to-one. However, in the case of a Denjoy transformation, the coding intervals can be taken to cover only the nonwandering set, and are therefore disjoint. This yields a conjugacy instead of a semiconjugacy.
\end{proof}

\begin{COR}
    Denjoy circle transformations are not variational.
\end{COR}

\begin{proof}
    By Poincar\'{e} recurrence, any invariant measure for a Denjoy transformation must be supported on its nonwandering set. Since restricted to this set, the system is topologically conjugated to a Sturmian shift, we conclude that it is uniquely ergodic and that the unique invariant measure is Kronecker. Thus, the metric entropy has 0 entropy at all scales. However, since there is a compact invariant set topologically conjugated to a Sturmian subshift, the semi-topological and topological entropies are both linear.
\end{proof}

\subsection{Geodesic flow on \texorpdfstring{$\mathbb{T}^2$}{}}
\label{sec:ergodic-decomp}

Another unexpected feature of slow entropy is the failure of additivity over ergodic decompositions. Let $\mu = \int_{\mathcal E} \nu d\hat{\mu}(\nu)$ be the ergodic decomposition of $\mu$, where $\mathcal E$ is the space of ergodic invariant measures and $\hat{\mu}$ is a probability measure on $\mathcal E$. For the classical entropy at exponential scale \cite[Theorem 9.6.2]{Viana_Oliveira_2016}, we have that

\begin{equation}
\label{eq:ergodic-decomp}
    h_\mu(f) = \int_{\mathcal E} h_\nu(f) d\hat{\mu}(\nu).
\end{equation} 

In this section, we explain that such a formula cannot hold for slow entropy, even when restricting so a fixed scale such as the polynomial scale.

It is well-known that the geodesic flow on $T^1\T^2$, the unit tangent bundle to $\T^2$, is not ergodic and has a smooth ergodic decomposition. Each ergodic component is diffeomorphic to $\T^2$ and corresponds to the unit speed linear flow in an irrational direction (the rational directions have measure 0, so we may omit them from the ergodic decomposition). Hence, at any scale, the ergodic components of the Haar measure on $T^1\T^2$ all have 0 entropy at all scales. The following Lemma shows that we can obtain a positive slow entropy by ``gluing'' several copies of Kronecker systems together in an interesting way.

\begin{LEM}
    \label{lemm:HaarGeod}
    If $\varphi_t : T^1\mathbb{T}^2 \to T^1\mathbb{T}^2$ is the geodesic flow on $\T^2$, then the topological and Haar slow entropy of $\varphi_t$ is 1 at polynomial scales $n_\chi(t) = t^\chi$.
\end{LEM}

\begin{proof}
    Observe that $\Isom(\R^2) \cong S^1 \ltimes \R^2$ acts simply transitively on $T^1\R^2$, and that $T^1\T^2$ is the quotient of $T^1\R^2$ by $\Z^2$. Furthermore, since the isometry group takes orbits of the geodesic flow to orbits of the geodesic flow, it follows that the geodesic flow is smoothly conjugated to homogeneous flow on $\Isom(\R^2) / \Z^2$ by a one-parameter subgroup of $\R^2 \subset \T^1\R^2$. If $\Theta \in \Lie(\Isom(\R^2))$ represents the generator of the subgroup $S^1$, and $X$ and $Y$ represent orthonormal generators of $\R^2$, then (up to choice of orientation), we have structure relations

    \[ [\Theta,X] = Y \qquad [\Theta,Y] = -X \qquad [X,Y] = 0. \]

    From this, one easily checks that in this basis $\ad(X)$ is
    
    \[\begin{pmatrix}
        0 & 1 & 0 \\ 0 & 0 & 0 \\ 0 & 0 & 0
    \end{pmatrix},\] so by \cite{ParabolicFlows}, it follows that the polynomial slow entropy of $\varphi_t$ is 1.
\end{proof}

We remark that this geodesic flow is also conjugated to the suspension of the affine map $(x,y) \mapsto (x,y+x)$, so we have the phenomenon for transformations as well.

\section{The slow entropy of some interval exchanges}
\label{sec:iets}
Interval exchange transformations or  IETs are piecewise isometries, with a finite number of discontinuities. Moreover, IETs preserve the orientation. These maps can be regarded as generalizations of rotations.
In this section, we will compute the metric slow entropy of 3-IETs. The computations for the metric entropy will occupy most of \Cref{sec:iets}. We will prove that for a large class of 3-IETs the metric slow entropy is 1 for the polynomial scale $a_\chi(n)=n^\chi$, see \Cref{thm:IETslowentropy}. 

\begin{THM}
    \label{thm:IETslowentropy}
    Let $\Delta \subset \R^3$ be the set $\{(x,y,z):x+y+z=1,\,x,y,z>0\}$. There exists a set $A\subset \Delta$ of Hausdorff dimension 2 such that if $g$ is a 3-IET determined by $\lambda \in A$, then $h_{\Leb,a_{\chi}}(g) = 1.$ 
\end{THM}

We will easily see that the topological slow entropy of the corresponding symbolic system is at most 1 with respect to the same scale, see \Cref{prop:TopEntridoc}. This combined with \Cref{thm:IETslowentropy} prove  \Cref{thm:3iet-main}.

\begin{RMK}
    We have used the convention that the limits appearing in the definitions of the functions $\delta_{f,\cdot,\chi}^S$ slow entropy are $\limsup$'s. In general these are not actual limits, and we rely on this choice several times in the proof. It would be interesting to make similar computations for the $\liminf$ definitions. It is already known that a gap may exist, and special attention is paid to this subtlety in \cite{BanKunWei2}. 
\end{RMK}

\subsection{Preliminaries of IET's}

We refer the reader to \cite{claynotes,Viana2006} for more details about interval exchange transformations (IET).

Let $\A$ be a collection of $d$ symbols and $\lambda\in\R^\A_{>0}$ be a vector of positive entries.
Given two bijective functions $\pi_t:\A \to \{1,\dots,d\}$, $\pi_b:\A \to \{1,\dots,d\}$, we obtain a permutation of the symbols in $\A$ defined by
\begin{equation*}
    \pi=\begin{pmatrix}
        \pi_t^{-1}(1) & \dots & \pi_t^{-1}(d)\\
        \pi_b^{-1}(1) & \dots & \pi_b^{-1}(d)\\
    \end{pmatrix}.
\end{equation*}
Let $I\subset \R$ be a bounded interval, closed on the left and open on the right, and denote the length of $I$ by $|I|$. From now on, we will assume that the left endpoint of $I$ is 0. The vector $\lambda$ and the permutation $\pi$ determine a partition $\{I_a\}_{a\in\A}$ where 
\begin{equation*}
I_{a}=\left[\sum_{\{b:\pi_t^{-1}(b)<\pi_t^{-1}(a)\}}\lambda_b,
 \sum_{\{b:\pi_t^{-1}(b)\leq \pi_t^{-1}(a)\}}\lambda_b
\right),
\end{equation*} and the interval $I = \left[0,|I|\right)$, where $ |I| = \sum_{A \in \A}\lambda_a.$

\begin{DEF}
     An {\bf interval exchange transformation} on $d$ intervals ({\bf $d$-IET}) $g \dfn g_{\lambda,\pi}\colon I \to I$ determined by a length vector $\lambda$ and permutation $\pi$ is the bijective map defined by 
    \begin{equation}
    \label{eq:IET}
        g(x) = x 
        +
        \sum_{\{b:\pi_b^{-1}(b)<\pi^{-1}_b(a)\}}\lambda_b
        -\sum_{\{b : \pi_t^{-1}(b) < \pi_t^{-1}(a)\}}\lambda_b,\quad \text{if }x \in I_a.
    \end{equation} 
\end{DEF}
Any IET is a measure preserving transformation with respect to the Lebesgue measure $\Leb$ on the interval $I.$
If for some $k<d$, the set $\{1,\dots,k\}$ is $\pi$ invariant, the IET is a concatenation of IETs with fewer intervals. If $\{1,\dots,k\}$ is not invariant for every $k<d$, we say that the permutation (and the IET) is {\bf irreducible}.

\subsection{Topological and semi-topological slow entropy of a class of \texorpdfstring{$d$-IETs}{}} Let $g=g_{\lambda,\pi}$ be a $d$-IET, and denote $D=\{\beta_1,\dots,\beta_{d-1}\}$ the set of discontinuities of $g$. Here, $\beta_i = \sum_{j=1}^{i}\lambda_j$. For convenience, denote $\beta_0=0$ and $\beta_d=|I|$.
The map $g$ has the {\bf idoc} property (infinite distinct orbit condition) if for all $n>1$, $D \cap g^{-n}(D) =\emptyset.$ 
Keane in \cite{Keane1975} proved that an idoc IET is minimal.

From now on, we will assume that $\pi$ is an irreducible permutation.
The set $I\backslash D$ has the $d$ continuity intervals of $g.$ By including the left endpoint, let $\P \dfn \{I_a\}_{a\in\mathcal{A}}$ be the {\bf natural partition of $g$.}

We have the following results regarding the number of Bowen balls. 
\begin{LEM}
\label{lemm:BBallsofIETs}
    Let $\e > 0$ be sufficiently small. Suppose that $g$ has the idoc property. There exists $n_0$ such that $B_g^{n+1}(x,\e) \subset \P^n(x)$ and $\P^n(x) \subset B_g^n(x,\e)$ for all $n\ge n_0$ and $x\in I.$
\end{LEM}
\begin{proof}
    Assume that $\e$ is small enough so that it has the following property: if $d(x,y)< \e$ but $y \not\in \P(x)$, then $d(g(x),g(y)) \ge \e$. Equivalently, if $d(x,y) < \e$ and $d(g(x),g(y)) < \e$, then $y \in \P(x)$. If no such $\e$ existed, then at least one of the points $\beta_i$ would be a removable discontinuity. Since $g$ has the idoc property, then  $||\P^n ||\dfn \max_{P\in\P^n}|P| \downarrow 0$. Let $n_0$ be such that $||\P^{n_0}||<\frac{\e}{d|I|}$.
    
    Since $\P^n = \bigwedge_{i=0}^{n-1}g^{-i}(P)$, then for all $n\geq n_0$ and $y \in \P^{n}(x)$, we have that  $d(g^i(x),g^i(y))<\e$ for $0\le i \le n-1$. This proves that $\P^n(x)\subset B^n_g(x,\e).$

    To see the other containment, suppose that $d(g^i(x),g^i(y)) < \e$ for $i = 0,\dots,n$. By the property establishing the smallness of $\e$, it follows that $g^i(x) \in \P(g^i(y))$ for $i = 0,\dots,n-1$.
\end{proof}

With the above, we can conclude the following about the topological entropy.

\begin{PROP}
\label{prop:TopEntridoc}
    Let $g_{\lambda,\pi}$ be $d$-IET. Suppose that it has the idoc property. Then $|\P^n|=dn$ and  $h_{top,a_\chi}=1$ with the polynomial scale $a_\chi(n)=n^\chi$.
\end{PROP}

\begin{proof}
    The idoc property implies that $\P^n=\bigwedge_{i=0}^{n-1}g^{-i}(\P)$ has exactly $dn$ atoms.
    Applying \Cref{lemm:BBallsofIETs}, 
    \begin{equation*}
        \delta_{g,I,\chi}(\e) = \limsup_{n\to \infty} \frac{d}{n^{\chi-1}}.
    \end{equation*}
    Then 
    \begin{equation*}
        h_{top,a_\chi}(g) = \lim_{\e \to 0}(\sup\{\chi:\delta_{g,I,\chi}(\e)>0\}) = 1.
    \end{equation*}
\end{proof}

Define $\e_n \dfn \min_{P\in \P^n}|P|$, the length of the smallest atom in the partition $\P^n.$ An IET is  {\bf linearly recurrent} if there exists a constant $C>0$ such that  for every $n\geq 1$, then $n\e_n \geq C.$

We have the following characterization of a big class of $d$-IETs.

\begin{PROP}
\label{prop:LRequalHomog}
    Suppose that $g$ is an idoc $d$-IET. The following are equivalent:
    \begin{enumerate}
        \item The Lebesgue measure is homogeneous, see \Cref{def:homog}.
        \item  $g$ is linearly recurrent.
    \end{enumerate}
\end{PROP}
\begin{proof}
    If we assume (1), fix $\e>0$ and  $c>0$ as in the definition of homogeneous measure. Then
    \begin{equation*}
        \frac{1}{c} \leq \frac{\Leb(B^n_g(x,\e))}{\Leb(B^n_g(y,\e))}\leq c
    \end{equation*}
    for all $n$ sufficiently large.
In particular, it follows that for every $P\in\P^n$,
\begin{equation*}
\frac{|P(y)|}{c} = \frac{\Leb(B^n_g(y,\e))}{c}\leq \e_n,
\end{equation*}
adding both sides of the inequality over $P\in\P^n$ we obtain
 $C\dfn \frac{1}{dc} \leq n\e_n$ for all $n$ sufficiently large. This proves (2).

Assume (2), and suppose by contradiction that (1) does not happen. So, for every $1>k>0$ there exists $n$ very large such that 
\begin{equation*}
\frac{\Leb(B^n_g(x,\e))}{\Leb(B^n_g(y,\e))}   < k. 
\end{equation*}
Without loss of generality, we can assume that $\e_n = \Leb(B^n_g(x,\e))$ and $\Leb(B^n_g(y,\e)) < \frac{1}{n}$. Then we have that $n\e_n = n\Leb(B^n_g(x,\e)) < k$. In particular, if $k=C$, which contradicts (2).
\end{proof}
We have the main result of this section:
\begin{COR}
\label{cor:semiequalstop1}
    Suppose that $g$ is an idoc, linearly recurrent $d$-IET. 
    Then, for $a_\chi=n^\chi$, we have that
    \begin{equation}
        h_{semi,Leb,a_\chi}(g)=h_{top,a_\chi} = 1
    \end{equation}
\end{COR}
\begin{proof}
     \Cref{thm:homogmeasures} and  \Cref{prop:LRequalHomog} imply that $h_{semi,\Leb,a_\chi}(g)=h_{top,a_\chi}(g)$. 
    \Cref{prop:TopEntridoc} implies that the equality is 1.
\end{proof}

We note that this does not say anything about the metric slow entropy of IETs. The class of linearly recurrent IET is uniquely ergodic by \cite[Theorem 1.2]{BoshernitzanCond}
then, computing the metric slow entropy of a linearly recurrent IET with respect to the Lebesgue measure will say whether such IET is variational. 
We also want to remark that the conditions in \Cref{cor:semiequalstop1} are satisfied by a set of parameters $\lambda$ of Hausdorff dimension $d$. This was noted by D. Robertson following his proof of \cite[Proposition 4]{ROBERTSON_2019} and the proof of \cite[Theorem 1.4]{JonChaika}.

\subsection{On 3-IETs}
 Any vector of positive entries $\lambda = (\lambda_A,\lambda_B,\lambda_C)$ and the {\bf symmetric} permutation $\pi=\begin{pmatrix}
    A & B & C\\
    C & B & A
\end{pmatrix}$ determines a 3-IET given by \Cref{eq:IET}, which in a simpler form is
\begin{equation*}
    g(x) = \left\{\begin{array}{ll}
        x+ \lambda_B+\lambda_C & \text{ if }x\in [0,\lambda_A),   \\
         x +\lambda_C -\lambda_A& \text{ if }x\in [\lambda_A,\lambda_A+\lambda_B),\\
         x-\lambda_A -\lambda_B & \text{ if }x\in [\lambda_A+\lambda_B,|I|).
    \end{array} \right.
\end{equation*}

It is common to think of 3-IETs over $I=[0,1)$, but we will consider 3-IETs with $I= [0,1 + \xi)$, $\A=\{A,B,C\}$, and the symmetric permutation $\pi=\begin{pmatrix}
    A & B & C\\
    C & B & A
\end{pmatrix}$.

Let  $\alpha\in (0,1)$ be an irrational number, and let $\left\{\frac{p_m}{q_m}\right\}_{m\geq 1}$ be the sequence of best approximations. Let $\| x \|$ be equal to $$\|x\| \dfn \min_{n \in \mathbb{Z}} |x - n |.$$ 

The number \(\alpha\) is {\bf badly approximable} if there exists \(C_{\alpha}>0\) depending on $\alpha$, satisfying 
\[q_{m+1}\leq C_{\alpha}q_m,\] 
for every \(m\in \mathbb{N}.\) 
Equivalently, $\alpha$ is badly approximable, if there exists $D_{\alpha}>0,$ such that 
\[\|m\alpha\|\geq \frac{D_\alpha}{m}\]
for every \(m\in \mathbb{N}.\) 

For every real number $\alpha$, we denote $S_\alpha$ the set
\begin{equation*}
    S_\alpha=\{\xi \in \R: \|\xi-j\alpha\|\geq  \frac{C_{\xi}}{q_n} \textit{\;for every \(-q_n<j<q_n\)}\}.
\end{equation*}

We have the following property regarding the Hausdorff dimension of the numbers.
\begin{LEM}
\label{lem:HausdDimoftwoSets}
The set of badly approximable real numbers has Hausdorff dimension one. For arbitrary $\alpha$, the set \(S_{\alpha}\) has Hausdorff dimension one.
\end{LEM}
\begin{proof}
For badly approximable numbers, this was proved in \cite[Theorem 3]{Schmidt}. Lastly,  \cite[Theorem 1]{Tseng} proved that  $\dim S_\alpha=1$ for every real number $\alpha$.
\end{proof}

\begin{LEM}
    Define $$X= \{(\alpha,\xi)\in (0,1)\times(0,1):\alpha \text{ is badly approximable and }\xi\in S_\alpha\}.$$ Then Hausdorff dimension of $X$, $\dim X = 2.$
    \label{lemma:HausdorffDim}
\end{LEM}

\begin{proof}

We apply the following theorem about the Hausdorff dimension of products 
\cite[Theorem 5.8 and Excercise 5.2]{Hausdorff}: Let $A,B$ be Borel subsets of a Euclidean spaces and let $Y \subset A\times B$ be such that for all $a\in A$, $\dim\{b\in B:(a,b)\in Y\}\ge d.$ Then 
\begin{equation*}
    \dim Y \ge \dim A + d.
\end{equation*}
Let $A$ be the set of badly approximable numbers, $B=(0,1)$, and replace $X$ into $Y$. The set $\{b:(a,b)\in X\}$ is the set $S_a.$ From \Cref{lem:HausdDimoftwoSets}, $\dim S_a =1.$

Therefore 
$\dim X \ge \dim A +1.$
Thus $\dim X = 2$, because $\dim A =1$ as stated in  \Cref{lem:HausdDimoftwoSets}.
\end{proof}
The proof of  \Cref{thm:IETslowentropy} follows from \Cref{prop:KeyProp} and \Cref{lem:ietEquiv}.
\Cref{prop:KeyProp} is the core of the argument and its proof will be postponed until \Cref{sec:CompSlowEntr}.

\begin{PROP}
\label{prop:KeyProp}
    Let $\xi \in (0,1)$ and let $\lambda = (\xi,\lambda_B,\lambda_C) \in \R^3_+$  such that $\lambda_B+\lambda_C = 1$. Define 
    \begin{equation*} 
        \alpha = 
            \left\{
            \begin{array}{cc}
         \lambda_C - \xi & \text{ if } \lambda_C >\xi,  \\
         1+ \lambda_C - \xi & \text{ if } \lambda_C < \xi.
            \end{array}
            \right.
    \end{equation*} 
    
    If $\alpha$ is badly approximable, and $\xi \in S_\alpha$,
    then the 3-IET $f \dfn f_{\lambda,\pi}$ determined in this way satisfies $h_{\Leb,a_\chi}(f)=1$ with polynomial scale $a_\chi(n) = n^\chi.$
\end{PROP}

\begin{LEM}
\label{lem:ietEquiv}
Let $\lambda \in \R^3_{>0}$, $c > 0$ and $\pi$ be a symmetric permutation. The 3-IETs $g \dfn g_{\lambda,\pi}$ and $h \dfn h_{c\lambda,\pi}$ are smoothly conjugated by the map $x \mapsto cx$. 
Moreover, any 3-IET $g \dfn g_{(\lambda_A,\lambda_B,\lambda_C),\pi}$ is measurably conjugated to its inverse $g^{-1} = g^{-1}_{(\lambda_C,\lambda_B,\lambda_A),\pi}$, by the hyperelliptic involution map $\iota(x) = |I| - x.$
\end{LEM}

\begin{proof}[Proof] 
    For $c>0$, let $f:[0,\sum_{\alpha} \lambda_\alpha)\to[0,c\sum_{\alpha} \lambda_\alpha) $ be the map $x \mapsto cx.$ Let $g$ be the 3-IET determined by the length vector $\lambda$ and the permutation $\pi$. Let $h$ be the 3-IET determined by the length vector $c\lambda$ and the permutation $\pi.$
    The map $f$ is a diffeomorphism that sends the Lebesgue measure on $[0,\sum_{\alpha} \lambda_\alpha)$ to the Lebesgue measure on $[0,c\sum_{\alpha} \lambda_\alpha)$.
    The IET $g$ maps by translation the segment $[\sum_{\pi_t(\alpha)<i}\lambda_\alpha,\sum_{\pi_t(\alpha)\le i}\lambda_\alpha)$ to the segment $[\sum_{\pi_b(\alpha)<4-i}\lambda_\alpha,\sum_{\pi_b(\alpha)\le 4-i}\lambda_\alpha)$. Thus the composition $f\circ g$ maps the segment $[\sum_{\pi_t(\alpha)<i}\lambda_\alpha,\sum_{\pi_t(\alpha)\le i}\lambda_\alpha)$ by stretching and translating to the segment $[c\sum_{\pi_b(\alpha)<4-i}\lambda_\alpha,c\sum_{\pi_b(\alpha)\le 4-i}\lambda_\alpha)$.

    Similarly, $h$ maps by translation the interval $[c\sum_{\pi_t(\alpha)<i}\lambda_\alpha,c\sum_{\pi_t(\alpha)\le i}\lambda_\alpha)$ to the interval $[c\sum_{\pi_b(\alpha)<4-i}\lambda_\alpha,c\sum_{\pi_b(\alpha)\le 4-i}\lambda_\alpha).$ In conclusion, for all $x \in [0,\sum_{\alpha}\lambda_\alpha)$, we have $f\circ g(x) = h \circ f(x).$ Proving that $g$ and $h$ are smoothly conjugated.

    Denote $I = [0,\lambda_A+\lambda_B+\lambda_C)$, and fix the permutation $\pi=\begin{pmatrix}
        A&B&C\\C&B&A
    \end{pmatrix}$. Let $g$ be the 3-IET determined by the length vector $\lambda=(\lambda_A,\lambda_B,\lambda_C)$ and the permutation $\pi.$ Let $g^{-1}$ be the inverse map of $g$. We let the reader verify that this is a 3-IET determined by the length vector $\lambda'=(\lambda'_A,\lambda'_B,\lambda'_C)\equiv (\lambda_C,\lambda_B,\lambda_A)$ and the permutation $\pi.$ 
    Let $\iota:I\to I$ be the map $x\mapsto|I|-x.$ 
    The map $\iota$ sends the open interval $(\sum_{\pi_*(\alpha) < i}\lambda_\alpha,\sum_{\pi_*(\alpha) \le i}\lambda_\alpha)$ to the open interval $(\sum_{\pi_*(\alpha) < 4-i}\lambda'_\alpha,\sum_{\pi_*(\alpha) \le 4-i}\lambda'_\alpha)$ by translation and order reversing, where $* \in \{t,b\}$.
    The composition $\iota\circ g$ maps the interval $(\sum_{\pi_t(\alpha) < i}\lambda_\alpha,\sum_{\pi_t(\alpha) \le i}\lambda_\alpha)$ to the interval $(\sum_{\pi_b(\alpha) < i}\lambda'_\alpha,\sum_{\pi_b(\alpha) \le i}\lambda'_\alpha)$ by a translation and order reversing.
    Also, the composition $g^{-1}\circ \iota$ maps the open interval $(\sum_{\pi_t(\alpha) < i}\lambda_\alpha,\sum_{\pi_t(\alpha) \le i}\lambda_\alpha)$ to the interval $(\sum_{\pi_b(\alpha) < i}\lambda'_\alpha,\sum_{\pi_b(\alpha) \le i}\lambda'_\alpha)$ by a translation and order reversing.  Thus, we have the equality $\iota\circ g = g^{-1} \circ \iota$ on the interior of the intervals $I_i$ for $i \in \{1,2,3\}$. 
    The equality does not occur at the left endpoints of the intervals $I_\alpha$, for example $\iota(g(0))=|I|-g(0)=|I|-\lambda_A-\lambda_B=\lambda_C$, but $g^{-1}(\iota(0))$ is not defined because $\iota(0)=|I|$ and $g^{-1}$ is not defined at $|I|.$ Thus the maps $g$ and $g^{-1}$ are measurable conjugated, since the map $\iota$ is a diffeomorphism that preserves the Lebesgue measure, and $\iota(g(x)) = g(\iota(x))$ except at finitely many $x \in I.$
\end{proof}

\subsection{Suspensions of a 3-IET and a rotation}

Let $(T,X,\mu)$ be measure preserving transformation and $f \colon X \to \R_{>0}$ an $L^1(\mu)$ function. 
Let $X^f$ be the quotient
\begin{equation*}
    X^f=\left \{(x,s):x \in X \text{ and } 0 \le s \le f(x) \right \}/\sim, 
\end{equation*}
with $(x,f(x)) \sim (T(x),0).$
The suspension flow determined by $T$ and $f$ is the flow
$T^f_t \dfn \Flow_t\colon X^f \to X^f$ defined by $\Flow_t(x,s) = (T^n(x),s+t-\sum_{i=0}^{n-1}f(T^ix))$, where $n \ge 0$ is such that
$f^{(n)}(x)  \leq s+t < f^{(n+1)}(x),$ where $f^{(n)}(x) $ is the \textbf{Birkhoff sum} 
\begin{align*}
    f^{(n)}(x) \dfn \sum_{i=0}^{n-1}f(T^ix).
\end{align*}

\subsubsection{A specific suspension for a 3-IET}
\label{sec:3IET}

Consider a 3-IET with $|I|=1+\xi$, and the conditions $\lambda_A=\xi < \lambda_B + \lambda_C= 1$. The suspension with the constant function 1 of this 3-IET is presented in \Cref{fig:3-IET}. This construction starts with a rectangle of length $1+\xi$ and height 1, by convenience assume that the left bottom corner is placed at the origin. The sides of the rectangle are identified by translation as follows
\begin{equation}
\label{eq:Identif3IET}
    \begin{split}
        [0,\lambda_C)\times \{0\} & \sim 
        [\lambda_A+\lambda_B,1+\xi)\times \{1\} \\
        [\lambda_C,\lambda_B+\lambda_C) \times \{0\} 
        & \sim
        [\lambda_A,\lambda_A+\lambda_B) \times \{1\} \\
        [1,1+\xi) \times \{0\} 
        & \sim
        [0,\lambda_A) \times \{1\} \\
        \{0\} \times [0,1) & \sim \{1+\xi\} \times [0,1).
    \end{split}
\end{equation}

\begin{figure}
\begin{subfigure}{.45\textwidth}
\centering
    \begin{tikzpicture}
        \def\l{4}
        \def\w{3}
        \def\la{1} 
        \def\lb{1.5}
        \def\lc{\l-\la-\lb}
        \draw (0,0) rectangle (\l,\w) ;  
        \foreach \lcolor/\llength/\xtop/\xbot/\ltext  in {
            blue/\la/0/\lb+\lc/A, 
            red/\lb/\la/\lc/B, 
            green/\lc/\la+\lb/0/C
        }{ 
            \draw [very thick, opacity=.8,\lcolor] (\xtop,\w) --++  (\llength,0) node[above left]{$\ltext$}  ;
            \draw [very thick, opacity=.8,\lcolor] (\xbot,0) node[below right]{$\ltext$} --++ (\llength,0)   ;
        }
        \draw [dashed] (\lb+\lc,0) -- ++ (0,\w) ;
        \foreach \x in {0,\lb+\lc}{ 
          \draw (\x, -.6cm+2pt) to (\x, -.6cm-2pt) ; 
          \draw (-.6cm-2pt,\x) to (-.6cm+2pt,\x) ;
        }
        \draw (0,-.6) to node [below,] {$1$}  (\lb+\lc,-.6) ;
        \draw (-.6,0) to node [left,] {$1$}  (-.6,\lb+\lc) ;
    \end{tikzpicture}
    \captionsetup{justification=centering}
    \caption{}
    \label{fig:3-IET}
\end{subfigure}%
\begin{subfigure}{.45\textwidth}
\centering
    \begin{tikzpicture}
        \def\la{1}
        \def\w{6}
        \def\l{3}
        \def\lb{1.5}
        \def\lc{\l-\lb}
        \draw (0,0) -- (0,\w) -- (\la,\w) -- (\la,\w/2) -- (\l,\w/2) -- (\l,0) --cycle  ; 
        \foreach \lcolor/\llengthtop/\llengthbot/\xtop/\xbot/\ytop/\ybot/\ltexttop/\ltextbot  in {
            red/\lb/\lb/\la/\lc/0.5*\w/0/B/B, 
            green/\l-\la-\lb/\lc/\la+\lb/0/0.5*\w/0/C_1/{}
        }{ 
            \draw [very thick, opacity=.8,\lcolor] (\xtop,\ytop) node[above right]{$\ltexttop$} --++  (\llengthtop,0)   ;
            \draw [very thick, opacity=.8,\lcolor] (\xbot,\ybot) node[below right] {$\ltextbot$} --++ (\llengthbot,0)   ;
        }
        \node [green,above] at (\la/2,\w) {$C_2$};
        \draw [green, very thick, opacity=0.8] (0,\w) -- (\la,\w);
        \draw [blue, dashed] (0,\w/2) -- node [above] {$A$} (\la,\w/2);
        \draw [green, dashed] (\lc-\la,-.3) -- (\lc-\la,.3);
        \node [green,below left] at (\lc-\la,0) {$C_1$};
        \node [green,below left] at (\lc,0) {$C_2$};
    \end{tikzpicture} 
    \captionsetup{justification=centering}
    \caption{}
    \label{fig:SuspRotation}
\end{subfigure}
\caption{This is case $\lambda_C>\xi$ in \Cref{eq:alphaintermsof3IET}. In \Cref{fig:3-IET}, the sides of a rectangle $[0,1+\xi)\times[0,1]$ are identified according to \Cref{eq:Identif3IET}. The vertical flow's first return time to $[0,1+\xi)$ is a 3-IET.
In \Cref{fig:SuspRotation}, suspension of a rotation on $[0,1]$ with $f =2\mathds{1}_{[0,\xi)}+\mathds{1}_{[\xi,1)}$. This can also be obtained by cutting over the dotted line in \Cref{fig:3-IET} and gluing the small piece over the segment $[1,1+\xi)\times\{1\}.$
}
\end{figure}
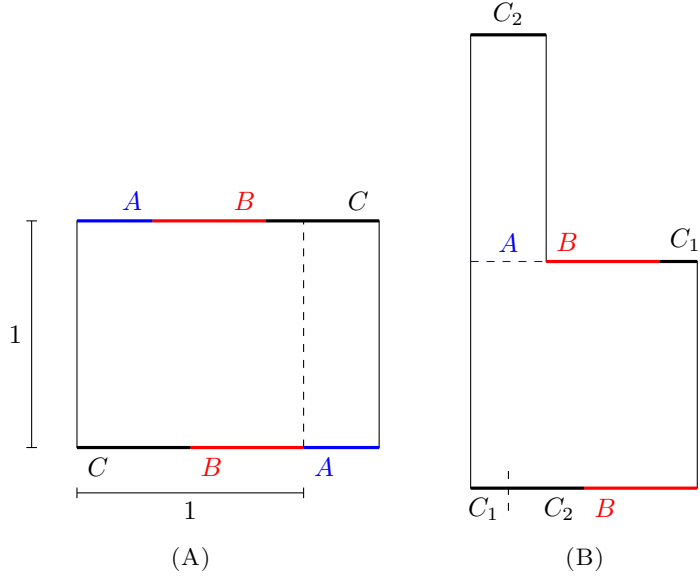

\begin{figure}
\begin{subfigure}{.45\textwidth}
\centering
    \begin{tikzpicture}
        \def\l{4}
        \def\w{3}
        \def\la{1.5} 
        \def\lb{1.3}
        \def\lc{\l-\la-\lb}
        \draw (0,0) rectangle (\l,\w) ;  
        \foreach \lcolor/\llength/\xtop/\xbot/\ltext  in {
            blue/\la/0/\lb+\lc/A, 
            red/\lb/\la/\lc/B, 
            green/\lc/\la+\lb/0/C
        }{ 
            \draw [very thick, opacity=.8,\lcolor] (\xtop,\w) --++  (\llength,0) node[above left]{$\ltext$}  ;
            \draw [very thick, opacity=.8,\lcolor] (\xbot,0) node[below right]{$\ltext$} --++ (\llength,0)   ;
        }
        \draw [dashed] (\lb+\lc,0) -- ++ (0,\w) ;
        \draw[|-|] (0,-0.6)--node[below,midway]{1}(\lb+\lc,-0.6);
        \draw[|-|] (-.6,0)--node[left,midway]{1} (-0.6,\w);
    \end{tikzpicture}
    \captionsetup{justification=centering}
    \caption{}
    \label{fig:3-IET2}
\end{subfigure}%
\begin{subfigure}{.45\textwidth}
\centering
    \begin{tikzpicture}
        \def\la{2}
        \def\w{6}
        \def\l{3}
        \def\lb{1.3}
        \def\lc{\l-\lb}
        \draw (0,0) -- (0,\w) -- (\la,\w) -- (\la,\w/2) -- (\l,\w/2) -- (\l,0) --cycle  ; 
    \draw [very thick, opacity=.8,green] (\la+\lb-\l,\w)  -- node[above right]{$C$} (\la,\w);
    \draw [very thick, opacity=.8,red] (0,\w)  -- node[above]{$B_2$} (\la+\lb-\l,\w) ;
    \draw [very thick, opacity=.8,red] (\la,\w/2)  -- node[above,pos=0.62]{$B_1$} (\l,\w/2) ;
    \draw [very thick, opacity=.8,red] (\lc,0)  -- (\l,0) ;
    \draw [very thick, opacity=.8,green] (0,0)  -- node[below,pos=0.62]{$C$} (\lc,0) ;
        \draw [blue, dashed] (0,\w/2) -- node [above] {$A$} (\la,\w/2);
        \draw [red, dashed] (\lc-\la+\l,-.3) -- (\lc-\la+\l,.3);
        \node [red,below left] at (0.95*\lc-0.95*\la+0.95*\l,0) {$B_1$};
        \node [red,below right] at (0.9*\l,0) {$B_2$};
    \end{tikzpicture} 
    \captionsetup{justification=centering}
    \caption{}
    \label{fig:SuspRotation2}
\end{subfigure}
\caption{
This is case $\lambda_C<\xi$ in \Cref{eq:alphaintermsof3IET}. In \Cref{fig:3-IET2}, the sides of a rectangle $[0,1+\xi)\times[0,1]$ are identified according to \Cref{eq:Identif3IET}. The vertical flow's first return time to $[0,1+\xi)$ is a 3-IET.
In \Cref{fig:SuspRotation2}, suspension of a rotation on $[0,1]$ with $f =2\mathds{1}_{[0,\xi)}+\mathds{1}_{[\xi,1)}$. This can also be obtained by cutting over the dotted line in \Cref{fig:3-IET2} and gluing the small piece over the segment $[1,1+\xi)\times\{1\}.$
}
\end{figure}
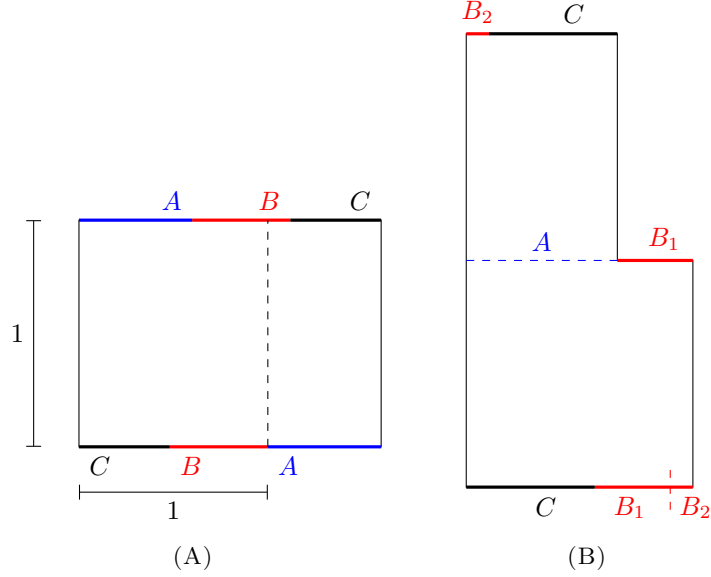

\subsubsection{Suspension flow of an irrational rotation and proof of \Cref{thm:IETslowentropy}}

Towards the computation of the slow entropy of the 3-IET mentioned in \Cref{sec:3IET}, we will compute the slow entropy of the vertical flow in the construction involving the suspension of the 3-IET in  \Cref{fig:SuspRotation}. An equivalent construction is the suspension of a rotation $T:[0,1)\to [0,1), x \mapsto x+ \alpha \mod 1$. The suspension function is $f = 2\mathds{1}_{[0,\xi)}+\mathds{1}_{[\xi,1)}.$ Computations for the first return time of the vertical flow to the segment $[0,1) \times \{0\}$ show that 
\begin{equation}
\label{eq:alphaintermsof3IET}
\alpha = 
    \begin{dcases}
        \lambda_C-\xi & \text{if } \lambda_C > \xi, \\
         1 + \lambda_C-\xi& \text{if } \lambda_C < \xi.
    \end{dcases}
\end{equation}

Although we are suspending a rotation, the 3-IET is still present in the first return map of the vertical flow to the segment $I=[0,1) \times \{0\} \cup [0,\xi) \cup \{1\} \subset [0,1)^f.$  
 \Cref{fig:3-IET} and \Cref{fig:SuspRotation}  are proof by picture of the case $\lambda_C>\lambda_A= \xi$. Because the rotation angle $\alpha$ in \Cref{fig:SuspRotation} is equal to the length of the segment labeled $C_1$, then $\alpha$ is equal to $\lambda_C-\lambda_A = \lambda_C-\xi$.
\Cref{fig:3-IET2} and \Cref{fig:SuspRotation2} are proof by picture of the case $\lambda_C<\lambda_A= \xi$. The rotation angle $\alpha$ of the suspension in \Cref{fig:SuspRotation2} is equal to the sum of the lengths of the segments labeled $C$ and $B_1$. The length of the latter is equal to $1-\lambda_A=1-\xi$, then $\alpha = 1+\lambda_C-\xi.$

Starting with $\alpha,\xi \in [0,1]$, the corresponding 3-IET is given by the length vector:
\begin{equation}
\label{eq:3IETfromalphaxi}
F_0(\alpha,\xi)
\dfn
(\lambda_A,\lambda_B,\lambda_C)=
\begin{dcases}
         (\xi,1-\alpha-\xi,\alpha+\xi)& \text{if }  \alpha+\xi < 1, \\
         (\xi,2-\alpha-\xi,\alpha+\xi-1)& \text{if }   \alpha+\xi > 1.
        \end{dcases}
    \end{equation}
The proof of \Cref{eq:3IETfromalphaxi} follows from similar computations and pictures as a verification of \Cref{eq:alphaintermsof3IET}.

\begin{proof}[Proof of \Cref{thm:IETslowentropy}]
Let $\Delta$ be the set $\{(x,y,z): x+y+z=1, \, x,y,z>0\}.$
The subset $X$ in \Cref{lemma:HausdorffDim} is of Hausdorff dimension 2. Let $F:[0,1]\times[0,1] \to \Delta$ be the function defined by $F(\alpha,\xi) = \frac{1}{1+\xi}F_0(\alpha,\xi).$ Since the function $F_0$ is linear  and of rank 2 in connected components, the Hausdorff dimension of the set $F_0(X)$ is 2. Additionally, the factor of $1/(1+\xi)$ preserves the Hausdorff dimension since it is the normalization factor, the $l_1$-norm of any point in the image of $F_0$ is $1/(1+\xi).$ This proves that the set $A \dfn F(X)$ is of Hausdorff dimension 2. 

Note that by \Cref{prop:KeyProp}, the metric slow entropy of the 3-IET determined by the vector $F_0(\alpha,\xi)$ is 1, by \Cref{lem:ietEquiv}, the 3-IET determined by $F_0(\alpha,\xi)$ and the 3-IET determined by $F(\alpha,\xi)=\frac{1}{1+\xi}F_0(\alpha,\xi)$ are measurably conjugated. Then, every 3-IET determined by a vector in $F(X)$ must have metric slow entropy 1 with $a_\chi = n^\chi.$
\end{proof}

\subsection{Computation of slow entropy and proof of \texorpdfstring{\Cref{prop:KeyProp}}{}}
\label{sec:CompSlowEntr}

We aim to compute the growth rate of the number of Hamming balls of time $R$ of the suspension flow $T^f_t$, because as we mentioned in the previous section, this suspension is equivalent to the suspension of a 3-IET with constant roof function 1.
Therefore, these Hamming balls' growth rate is the same as the Hamming balls' of the suspended 3-IET.

Let $\alpha$ and $\xi$ as before and 
$f=d_1\mathds{1}_{[0,\xi)}+d_2\mathds{1}_{[\xi,1)}$. The
 conditions on \(\alpha\) and \(\xi\) will be given later. Given a subset $A \subset \mathbb{T}$, we use \(A^f\) to denote the set 
\(\{(y,t)\in \mathbb{T}^f: y\in A\}\). We also let
\(\lambda^f\) denote the normalized Lebesgue measure restricted to $\mathbb{T}^f$. Let \(M=\max\{d_1,d_2,\frac{1}{d_1},\frac{1}{d_2}\}\) and \(e=|d_1-d_2|.\)
The special situation in which we prove \Cref{prop:KeyProp} is by setting \(d_1=2\) and \(d_2=1.\) 
Given $R>0$ very large and $\e>0$, fix a generating partition 
$\P$ such that all atoms in the partition are squares with length between \(1/2k\) to \(1/k\) for some large value \(k\) to be specified later. We want to show that $S_{T^{f},\P}(\e,R)\geq CR$ holds for some constant $C$, hence 

$$\displaystyle\limsup_{R\rightarrow \infty} \frac{S_{T^{f},\P}(\e,R)}{R^{\chi}}=\infty$$ for any $\chi<1$. This gives us a lower bound for the metric slow entropy, i.e. $h_{\mu,a_\chi}(T^f) \geq 1.$
 
From now on, we will assume that $\alpha$ is badly approximable; otherwise, we will specify.

To prove \Cref{prop:KeyProp}, we combine several lemmas which describe recurrence properties for the base circle rotation $R_\alpha$. In fact, this will allow us to compute the slow entropy of some other special flows with piecewise constant roof functions (see \Cref{prop:loc-const-roof}). We therefore state a few lemmas (\Cref{Prop Ratner} - \Cref{Computation of gap}), which we use to prove \Cref{prop:KeyProp}, delaying their proof until later in this section. 

The following proposition uses similar ideas in \cite[Lemma 4]{fraczek2007mild}, and is a Ratner-type property for the special flows we consider.
\begin{PROP}
\label{Prop Ratner}
Suppose \(\alpha\in (0,1)\) is badly approximable, and \(\xi\in S_{\alpha}\). There exists $\kappa,c>0$ and a finite set $V\subset {\mathbb{R}\setminus\{0\}}$ such that for all large enough $m$, and $ x,y \in \mathbb{T}$ with 
$$\frac{C_{\xi}}{2q_{m+1}}\leq \|x-y\|<\frac{C_{\xi}}{2q_m},$$ for some time $N$ where $$ q_m \leq N \leq cq_{m+1}+2q_m,$$  the following set $$\Lambda=\{n\in [0,N]\cap \mathbb{Z}:  f^{(n)}(x)-f^{(n)}(y)\in V\}$$ has cardinality larger than $\kappa N$.   
\end{PROP}

\Cref{Prop Ratner}  shows that in some definite proportion, the Birkhoff sum $f^{(n)}$ of nearby points will differ. Moreover, the next lemma shows that the splitting phenomenon occurring in the Birkhoff sums implies the splitting time of two orbits in special flow space by a definite proportion. 
\begin{LEM}
\label{Divergence in flow}
Let $\tilde{x}=(x,s)$, $\tilde{y}=(y,s')$ be two elements in the same atom of partition. If $$\frac{C_{\xi}}{2q_{m+1}} \leq\|x-y\|<\frac{C_{\xi}}{2q_{m}}$$ for some $m$, there exists some time $T$ and constant \(D\) where
\(T \leq Dq_m\) and $\kappa'>0$ such that the following set
$$\tilde{\Lambda}=\{t<T: T_t^f{\tilde{x}} \text{\;and\;} T_t^f{\tilde{y}} \text{\;are not in the same atom }\}$$ has measure $|\tilde{\Lambda}|$ larger than $\kappa' T$.
\end{LEM}

 Let $\tilde{x}=(x,s)$ be an element in $\mathbb{T}^f$. Consider a Hamming ball $B_{T^f,\P}^R(\tilde{x},\e)\eqqcolon B^R(\tilde{x},\e)$ centered at $\tilde{x}$, then for any $\tilde{y}=(y,s')\in B^R(\Tilde{x},\e)$ we want to show there exists $n_1$ and $n_2$ such that $\|x+n_1\alpha-(y+n_2\alpha)\|<H/R$ for some constant $H$.

\begin{LEM}
\label{Main lemma}
There exists a constant $H,$ such that for any $\tilde{y}=(y,s')\in B^R(\Tilde{x},\e)$, there are $t_0,s_0,s_0'\in(0,R)$, $n_1=n_1(t_0,x,s)\in \mathbb{Z}$ and $n_2=n_2(t_0,y,s')\in \mathbb{Z}$ satisfying $\|x+n_1\alpha-(y+n_2\alpha)\|<H/R$, where 
$T_{t_0}^f\Tilde{x}=(x+n_1\alpha,s_0)$ and $T_{t_0}^f\Tilde{y}=(y+n_2\alpha,s_0')$ are in the same atom of $\P$.
\end{LEM}

From \Cref{Main lemma}, for every $\tilde{y}=(y,s')\in B^R(\Tilde{x},\e)$, we obtain $\|x-y-(n_2-n_1)\alpha\|\leq H/R.$ Therefore, $ \tilde{y}$ is in $$\displaystyle\bigcup_{j=-|n_2-n_1|}^{|n_2-n_1|} \left[x+j\alpha-\frac{H}{R},x+j\alpha+\frac{H}{R}\right]^f.$$ 
To get a lower bound of $S_{T^{f},\P}(\e,R)$, we will compute a uniform upper bound for the Lebesgue measure of $B^R(\Tilde{x},\e)$, it suffices to find an upper bound for $|n_2-n_1|$ for all $\tilde{x}$ and $\tilde{y}\in B^R(\Tilde{x},\e)$.
\Cref{Computation of gap} below, summarizes this idea.

\begin{LEM}
\label{Computation of gap}
There exists a constant $G>0$ depending only on $\alpha$ and $\xi$ such that
$|n_2-n_1|\leq G$, where $n_1$ and $n_2$ are the integers in \Cref{Main lemma}.
\end{LEM}

Let \(\mathcal{D} \subset [0,1]\) be the set consisting of irrational numbers \(\alpha\) such that there exists \(C_1>0\) depending on $\alpha$, satisfying
\[q_{m+1}\leq C_1q_m \log^2 q_m \] 
for every \(m.\) In particular $\mathcal{D}$ contains badly approximable irrational numbers.

\begin{PROP}
\label{prop:loc-const-roof}
If $\alpha\in \mathcal{D}$,  \(\xi \in \mathbb{T}\) and we consider the roof function \(f=d_1\mathds{1}_{[0,\xi)}+d_2\mathds{1}_{[\xi,1)}\), the metric slow entropy of the special flow system is at most 1 for scale \(a_{\chi}(t)=t^{\chi}\). 
\end{PROP} 
 \begin{proof}[Proof of \Cref{prop:KeyProp}]

From \Cref{Main lemma}, for $R$ large enough and $\tilde{y}\in B^R(x,\e)$, there are $n_1,n_2$ natural numbers, such that
$$\tilde{y}\in\displaystyle\bigcup_{j=-|n_2-n_1|}^{|n_2-n_1|} \left[x+j\alpha-\frac{H}{R},x+j\alpha+\frac{H}{R}\right]^f.$$ 
From \Cref{Computation of gap}, there is an upper bound $G$ for $|n_2-n_1|$ independent of $\tilde{y}.$
Thus $\Leb(B^R(x,\e))\le \frac{4GH}{R}.$
In other words, for any large $R>0$, the measure of each Hamming ball is bounded above by $C'/R$ for some constant $C'$. Thus, there exists a constant $C>0$ independent of $R$ and $\e$ such that $S_{T^{f},\P}(\e,R)\geq CR$. This shows the slow entropy $\chi$ for this special flow is at least 1 using time scale $a_{\chi}(t)=t^{\chi}$. 

 Since $\alpha \in \mathcal{D}$ ($\alpha$ is badly approximable), applying \Cref{prop:loc-const-roof} the upper bound of the metric slow entropy is 1.
\end{proof}

The remaining of this paper is dedicated to prove \Cref{Prop Ratner}, \\
\Cref{Divergence in flow}, 
\Cref{Main lemma}, 
\Cref{Computation of gap}, 
and \Cref{prop:loc-const-roof}.

\begin{proof} [Proof of \Cref{Prop Ratner}]
By the definition of $f^{(n)}$, and assuming that $0<x<y<1$, it follows:
\begin{equation*}
 f^{(n)}(x)-f^{(n)}(y)=e\sum_{j=0}^{n-1}\mathds{1}_{(x,y]}(\xi-j\alpha)-e\sum_{j=0}^{n-1}\mathds{1}_{(x,y]}(-j\alpha).  
\end{equation*}  
By assumption on \(\xi\), the interval of \((x,y]\) can be crossed by the orbit of \(\xi\) (and the orbit of \(0\)) for at most \(\frac{cq_{m+1}+2q_{m}}{q_{m}}+1\leq [cC_{\alpha}]+4\eqqcolon U\) times. Hence, choose \[V=\{ne: -U\leq n \leq U \text{ and } n\neq 0\}.\]
Taking \(\kappa= \frac{1}{2cC_{\alpha}+4}\),
assume for all \(0\leq n\leq q_m, \) \(f^{(n)}(x)-f^{(n)}(y)\in V\), then we are done. Suppose \(n_0\) is the minimum number such that \[f^{(n_0)}(x)-f^{(n_0)}(y)=0. \]
There exists \(I_0\), where \(I_0\) is the minimum time such that \(\bigcup\limits_{j=0}^{I_0}R_{\alpha}^{j}(x,y]\) covers \(0\) or \(\xi\). Because \(\alpha\) is badly approximable, let \(c\) be the constant such that \(I_0\leq c q_{m+1}.\) Here \(c\) is a constant depending only on \(\alpha\) and \(\xi.\) Assume without loss of generality, \(\xi \in (x+I_0\alpha,y+I_0\alpha].\) Assume at time \(J_0,\) 
\(0\in (x+(I_0+J_0)\alpha,y+(I_0+J_0)\alpha]. \)
This means
\[\|\xi-(x+I_0\alpha)\|\leq \frac{C_{\xi}}{2q_m}\]
and
\[\|x+(I_0+J_0)\alpha\|\leq \frac{C_{\xi}}{2q_m}.\]
By triangle inequality, we obtain 
\(\|\xi+J_0\alpha\|\leq \frac{C_{\xi}}{q_m}.\)
By definition of \(\xi\), this means \(J_0\geq q_m.\)
Therefore, when \(n\in[I_0+1,I_0+J_0]\), \[f^{(n)}(x)-f^{(n)}(y)\in V.\]
Hence, choose \(N\) to be the time \(n_0+I_0+q_m\), we complete the proof.
\end{proof}

\begin{proof}[Proof of \Cref{Divergence in flow}]
From \Cref{Prop Ratner}, we can find such constants $\kappa$, and $N$, and the interval $\Lambda=[a,b]\cap \mathbb{N}$. For any time $t_0\in  [f^{(a)}(x)-s,f^{(b)}(x)-s]$, 
$T_{t_0}^f{\tilde{x}}=(x+n_1\alpha,s_0)$, 
$T_{t_0}^f{\tilde{y}}=(y+n_2\alpha,s_0')$.
Then 
$$t_0=f^{(n_2)}(y)+s_0'-s'= f^{(n_1)}(x)+s_0-s.$$ 
If $T_{t_0}^f(\tilde{x})$ 
and $T_{t_0}^f(\tilde{y})$ are in the same atom, we can obtain
$|s_0-s_0'|<1/k$.
Since $\tilde{x} $ and $\tilde{y}$ are also in a same atom, we have
$$|f^{(n_2)}(y)-f^{(n_1)}(x)|\leq 2/k.$$
But from \Cref{Prop Ratner}, there exists $p\in V$ such that $$f^{(n_1)}(x)=f^{(n_1)}(y)+p.$$
Thus, by combining the above two equations, we deduce:
$$|f^{(n_2)}(y)-f^{(n_1)}(y)-p|\leq 2/k.$$
Since $V$ is a fixed finite set, we can assume 
$2/k\ll \min{\{|q|: q\in V}\}$, so $n_2\neq n_1$. 
Because $T_{t_0}^f(\tilde{x})$ and $T_{t_0}^f(\tilde{y})$ are in the same atom and $f\geq \frac{1}{M}$, we have the following relations:
\begin{equation}
    \frac{D_{\alpha}}{|n_2-n_1|}\leq\|(n_2-n_1)\alpha\|<\frac{2}{k},
\end{equation}
and
\begin{equation}
|n_2-n_1|\leq (\max{\{|q|: q\in V\}}+1)M.    
\end{equation}

The above two equations are contradictory since we can choose $k$ to be arbitrarily large.
Therefore, for any such time 
$t_0\in [f^{(a)}(x)-s,f^{(b)}(x)-s] $,  $T_{t_0}^f(\tilde{x})$ and $T_{t_0}^f(\tilde{y})$ are not in the same atom.
Choosing $T=f^{(b)}(x)-s$, and $\kappa'=\kappa/100M$. We know that 
\[T\leq (N+1) M \leq (cC_{\alpha}+3)q_m M  .\]
Hence, choose \(D=(cC_{\alpha}+3)M\), we complete the proof.
\end{proof}

\begin{proof}[Proof of \Cref{Main lemma}]
 By our choice of generating partitions, the set of times for which $T_t^f\tilde{x}$ and $T_t^f\tilde{y}$ belong to the (closure of) the same partition element is a union of closed intervals. We may therefore write $(0,R)$ as a disjoint union of intervals
 $\displaystyle \bigcup_{j=1}^J I_j$, 
 where for each $I_j$, $T_t^f(\tilde{x})$ and 
 $T_t^f(\tilde{y})$ are either in the same atom of $\P$, or they are in different atoms of $\P$ and the intervals $I_i$ are maximal among such choices. In particular, either for all odd values or all even values of $i$, the interior of $I_i$ consists of matching times for $\tilde x$ and $\tilde y$.\\
{\bf  Case 1.} If $|J|=1$ or 2, $T_t^f(\tilde{x})$ and $T_t^f(\tilde{y})$ stays in the same atom all time in an interval of length at least
$(1-\e)R$.
Let $x$ and $y$ be the corresponding first coordinates when they stay in the same atom for the first time. We know that 
\begin{equation}
\label{Birkhoff estimate}
   f^{(n)}(x)-f^{(n)}(y)=e\sum_{j=0}^{n-1}\mathds{1}_{(x,y]}(\xi-j\alpha)-e\sum_{j=0}^{n-1}\mathds{1}_{(x,y]}(-j\alpha). 
\end{equation}

Since $f$ is bounded between \(d_1\) and \(d_2\), $R/2M<N<2MR$ where $N$ is an upper bound of the number of terms in \Cref{Birkhoff estimate}.   Let $d=\|x-y\|$, as in the proof of \(\Cref{Prop Ratner}\), there exists a constant $N(d)$ where $N(d)\leq \frac{K}{d}$ for some constant $K$, s.t. $\mathbb{T}\subset\displaystyle\bigcup_{j=1}^{N(d)}(x+j\alpha,y+j\alpha]$.
By assumption, $(x,y]$ cannot cross discontinuity points before \(N/2\) times of rotation when the orbit of $\tilde{x}$ and $\tilde{y}$ stay in the same atom. Hence, $N(d)\geq R/4M$.
Therefore, there exists a constant $H_1$ s.t. $d\leq H_1/R$.\\
{\bf  Case 2.} If $|J|>2$, we regroup those intervals in the following way. 

Let $t_0$ be the first time when $T_{t_0}(\tilde{x})$ and $T_{t_0}(\tilde{y})$ stay in the same atom. Denote $\tilde{I}_0=\{t: t<t_0\}$, note $\tilde{I}_0$ could be an empty set.
Then let $T_{t_0}^f(\tilde{x})=(x_0,s_0)$, $T_{t_0}^f(\tilde{y})=(y_0,s_0')$. Since they are in the same atom, there exists an integer $m$ s.t. $$\frac{C_{\xi}}{2q_{m+1}}\leq \|x_0-y_0\|<\frac{C_{\xi}}{2q_m}\leq 1/k.$$ 
Applying \Cref{Divergence in flow}, we can find the corresponding time $T$. By definition, $t_0\in I_{j_0}$ for some $j_0$. Let $j_1>j_0$ be the minimal natural number (if exists) such that the following interval $\displaystyle\bigcup_{l=j_0}^{j_1-1} I_{l}$,
has a length larger than $T$, and $I_{j_1}$ is the interval when $T_{t}(\tilde{x})$ and $T_{t}(\tilde{y})$ are in the same atom. If such $j_1$ exists, denote $$\tilde{I}_1:=\displaystyle\bigcup_{l=j_0}^{j_1-1} I_{l}.$$
Otherwise, denote $\tilde{I}_1=[t_0,R)$. When $\displaystyle \cup_{j=0}^{1} \tilde{I}_{j} \neq (0,R)$, we know at the starting time $t_1$ of $I_{j_1}$, $T_{t_1}(\tilde{x})$ and $T_{t_1}(\tilde{y})$ are in the same atom, we can then repeat the above procedure to find $\tilde{I}_2$.
  Inductively, we can get $$(0,R)=\displaystyle \bigcup_{j=0}^{L} \tilde{I}_{j}$$ where $L+1$ is the number of such intervals.
By assumption, \(L\geq 1.\) 

If $L=1$, for the first time when the orbit of $\tilde{x}$ and $\tilde{y}$ are in the same atom, let $m$ be the number s.t. the distance of the first coordinate is between $\frac{C_{\xi}}{2q_{m+1}}$ and $\frac{C_{\xi}}{2q_m}$, then we can obtain $R<T$. Otherwise, from \Cref{Divergence in flow}, there will be a constant $\kappa'$ which is independent of $R$ and $\e$, s.t. the total time when those two orbits are not in the same atom is larger than $\kappa' R$. Since we can take $\e<\kappa'$ because $\kappa'$ is a fixed constant,  this is a contradiction. 
Hence, $$\frac{1}{D q_m}\leq\frac{1}{T}\leq \frac{1}{R}.$$
Therefore, the distance of first coordinate if less than \(\frac{H_2}{R}\) for some constant \(H_2.\)

If $L>1$, the sum of the length of $\tilde{I}_{j}$ where $j=1,\dots,L-1$ is less than $\e R/\kappa'$. Therefore, the last interval is of length larger than $(1-\e /\kappa') R$. Using a similar argument, there exists a constant $H_3$ satisfying 
the distance of the first coordinate is less than $H_3/R$.
Finally, taking $H=\max\{H_1,H_2,H_3\}$, we finish the proof.   
\end{proof}

\begin{proof}[Proof of \Cref{Computation of gap}]

 From the definition of  $n_1$ and $n_2$, we know 
 $$|f^{(n_1)}(x)-f^{(n_2)}(y)|\leq |f^{(n_1)}(x)-f^{(n_2)}(y)+s_0-s_0'|+|s_0-s_0'|\leq G_1$$ for some constant $G_1.$
 Since $n_1\leq 2M R$, we know there exists constant $G_2$ s.t. 
 $$|f^{(-n_1)}(x+n_1\alpha)-f^{(-n_2)}(y+n_2\alpha)|\leq G_2,$$ this is because the interval $[x+n_1\alpha,y+n_2\alpha)$ only cross discontinuity points finitely many times.
 Using triangle inequality, cocycle identity, and $f\geq \frac{1}{M}$, we get 
 \begin{equation}
   \begin{aligned}
        G_1&\geq|f^{(n_1)}(x)-f^{(n_2)}(y)|\\&=|f^{(-n_1)}(x+n_1\alpha)-f^{(-n_2)}(y+n_2\alpha)+f^{(n_2-n_1)}(y+n_2\alpha)|\\&\geq 
 |f^{(n_2-n_1)}(y+n_2\alpha)|-G_2\\
        &\geq \frac{|n_2-n_1|}{M}-G_2
   \end{aligned}
 \end{equation}
Thus we can obtain $|n_2-n_1|\leq (G_2+G_1)M \eqqcolon G$.
\end{proof}

\begin{proof}[Proof of \Cref{prop:loc-const-roof}]
Choose \(\delta=\frac{\eta}{2}\).
Let $Z_n$ be the subset
\begin{equation*}
\displaystyle \bigcup_{j=-q_n}^{q_n}\left\{x\in\mathbb{T}: x+j\alpha \in (-2q_n^{-1-\delta},2q_n^{-1-\delta}) \cup (-2q_n^{-1-\delta}+\xi,\xi+2q_n^{-1-\delta})\right\}. 
\end{equation*}
Since \( \sum_{n=1}^{\infty} \frac{1}{q_n^{\delta}}<\infty,\) when \(N\) is large enough, \(\lambda\left(\displaystyle \cup_{n=N}^{\infty}Z_n \right)<\e/4.\)

Denote \(\tilde{X}=\left(\displaystyle \cup_{n=N}^{\infty}Z_n \right)^c,\) then \(\lambda^f \left(\tilde{X}^f \right)>1-\e/2\).  Define \[\Omega_{\kappa}=\{ \tilde{y}\in \mathbb{T}^f: d(\tilde{y}, \partial\P)<\kappa\}.\] There exists \(0<\tilde{\e}<\frac{\e}{100Mk}\), such that \(\lambda^f(\Omega_{\tilde{e}})<\e/4.\) 
There exists some integer \(l\) such that \(q_{l}\leq R <q_{l+1},\) choose \(R\) to be a large value such that \(l\geq 2N\) and \(q_{l}^{-1}<\tilde{e}/100M.\) Take \(A=\tilde{X}\cap \Omega_{\tilde{e}}.\) 
Note that for \(\tilde{x}=(x,s)\in A\) the Hamming ball \(B^R(\tilde{x},\e)\) centered at \(\tilde{x}\)  containing \[\left\{(y,t)\in \mathbb{T}^f: y\in \left(-\frac{1}{q_{l+1}^{1+\delta}}+x,x+\frac{1}{q_{l+1}^{1+\delta}}\right), t\in (s-\tilde{e},s+\tilde{e})\right\}.\] Hence, \(S_{T^f,\P}(\e,R)\leq C R^{1+\delta}\log^{2+2\delta} R\) for some constant \(C\) independent of \(R\). It follows that   \(\delta_{T^f,\P,\chi}^S (\e)=0\) when \(\chi=1+\eta.\)
\end{proof}

\bibliography{bib}
\bibliographystyle{alpha}

\end{document}